\documentclass[11pt,reqno,oneside]{amsart}

\usepackage{graphicx,subfigure,threeparttable}
\usepackage{amssymb}
\usepackage{epstopdf}

\usepackage[a4paper, total={6in, 9in}]{geometry}

\usepackage{booktabs} 
\usepackage{subfig} 

\usepackage{amsfonts}

\usepackage{amsmath,amsthm,mathrsfs,amssymb,cite}
\usepackage[usenames]{color}

\newtheorem{thm}{Theorem}[section]

\newtheorem{lem}{Lemma}[section]

\theoremstyle{definition}

\newtheorem{alg}{Algorithm}

\theoremstyle{remark}

\newtheorem{rem}{Remark}[section]
\numberwithin{equation}{section}

\numberwithin{equation}{section}

\newcounter{saveeqn}

\title[A novel quantitative inverse scattering scheme]{A novel quantitative inverse scattering scheme using interior resonant modes}

\author{Youzi He}
\address{Department of Mathematics, Hong Kong Baptist University, Kowloon, Hong Kong SAR, China}
\email{yolanda19he@gmail.com}

\author{Hongyu Liu}
\address{Department of Mathematics, City University of Hong Kong, Kowloon, Hong Kong SAR, China}
\email{hongyliu@cityu.edu.hk; hongyu.liuip@gmail.com}

\author{Xianchao Wang}
\address{
School of Mathematics, Harbin Institute of Technology, Harbin,  China}
\email{xcwang90@gmail.com}

\date{} 
\begin{document}

\maketitle

\begin{abstract}
	This paper is devoted to a novel quantitative imaging scheme of identifying impenetrable obstacles in time-harmonic acoustic scattering from the associated far-field data. The proposed method consists of two phases. In the first phase, we determine the interior eigenvalues of the underlying unknown obstacle from the far-field data via the indicating behaviour of the linear sampling method. Then we further determine the associated interior eigenfunctions by solving a constrained optimization problem, again only involving the far-field data. In the second phase, we propose a novel iteration scheme of Newton's type to identify the boundary surface of the obstacle. By using the interior eigenfunctions determined in the first phase, we can avoid computing any direct scattering problem at each Newton's iteration. The proposed method is particularly valuable for recovering a sound-hard obstacle, where the Newton's formula involves the geometric quantities of the unknown boundary surface in a natural way. We provide rigorous theoretical justifications of the proposed method. Numerical experiments in both 2D and 3D are conducted, which confirm the promising features of the proposed imaging scheme. In particular, it can produce quantitative reconstructions of high accuracy in a very efficient manner.

	\medskip

	

	
\end{abstract}

\section{Introduction}


In this article, we are concerned with the inverse acoustic scattering problem of reconstructing an impenetrable obstacle by the associated far-field measurement. To begin with, we present the mathematical setup of the inverse scattering problem for our study.

Let $k\in\mathbb{R}_+$ be the wavenumber of a time harmonic wave and $D\subset \mathbb{R}^m, m=2,3$ be a bounded domain with a Lipschitz-boundary $\partial D$ and a connected complement $\mathbb{R}^m \setminus \overline{D}$, which signifies the target obstacle in our study. We take the incident field $u^i$ to be a time harmonic plane wave of the form
\[
u^i:=u^i(x, d, k)=\mathrm{e}^{\mathrm{i}k x\cdot d}, \quad  x \in \mathbb{R}^m,
\]
where $\mathrm{i}:=\sqrt{-1}$ is the imaginary unit, $d \in \mathbb{S}^{m-1}$ is the direction of propagation and $\mathbb{S}^{m-1}:=\{x \in \mathbb{R}^m: |x|=1 \}$ is the unit sphere in $\mathbb{R}^m$. Let $u^s$ and $u:=u^i + u^s$ signify the scattered and total wave fields, respectively. The forward scattering problem is described by the following Helmholtz system:
\begin{equation}\label{eq:forwardpr}
	\left\{
	\begin{array}{ll}
		\Delta u+k^2 u =0  & \text{in} \  \mathbb{R}^m \setminus \overline{D}, \medskip \\
		\displaystyle \mathcal{B}(u) = 0  & \text{on} \  \partial D, \medskip \\
		\displaystyle \lim\limits_{r\to\infty} r^{\frac{m-1}{2}}\left(\frac{\partial u^s}{\partial r}-\mathrm{i}ku^s\right)=0, & r=|x|,
	\end{array}
	\right.
\end{equation}
where the last limit is known as the Sommerfeld radiation condition which holds uniformly in $\hat x:=x/|x|\in\mathbb{S}^{m-1}$ and characterizes the outgoing nature of the scattered field. In \eqref{eq:forwardpr}, $\mathcal{B}(u)=u$ or $\mathcal{B}(u)=\partial u/\partial\nu$, respectively, signify the physical scenarios that the obstacle $D$ is sound-soft or sound-hard. Here and also in what follows, $\nu\in\mathbb{S}^{m-1}$ denotes the exterior unit normal vector to $\partial D$. The well-posedness of the scattering system \eqref{eq:forwardpr} can be conveniently found in \cite{LSSZ}. There exists a unique solution $u\in H_{loc}^1(\mathbb{R}^m\backslash\overline{D})$, and it admits the following asymptotic expansion \cite{CK19}:
\begin{equation*}
	u^s(x,d,k)=\frac{\mathrm{e}^{\mathrm{i}k|x|}}{|x|^{\frac{m-1}{2}}} \bigg\{u^{\infty}(\hat{x},d,k)+\mathcal{O}\bigg(\frac{1}{|x|} \bigg) \bigg\} \quad \text{as} \ |x|\to\infty,
\end{equation*}
which holds uniformly for all directions $\hat{x}:=x/|x|\in \mathbb{S}^{m-1}$. The complex-valued function $u^{\infty}$ defined on the unit sphere $\mathbb{S}^{m-1}$ is known as the far-field pattern of the scatter field $u^s$, which encodes the information of the scattering obstacle $D$. The inverse scattering problem of our concern is to recover $D$ by knowledge of $u^\infty$; that is,
\begin{equation}\label{eq:ipp1}
	\mathcal{F}(D)=u^{\infty}(\hat{x},d,k), \quad (\hat{x},d,k)\in \mathbb{S}^{m-1} \times \mathbb{S}^{m-1} \times V,
\end{equation}
where $V$ is an open interval in $\mathbb{R}_{+}$, and $\mathcal{F}$ is an abstract operator defined by \eqref{eq:forwardpr}. Though the forward scattering problem \eqref{eq:forwardpr} is linear, it can be straightforwardly verified that the inverse problem \eqref{eq:ipp1} is nonlinear.

The inverse scattering problem \eqref{eq:ipp1} is prototypical model of fundamental importance for many scientific and industrial applications including radar/sonar, medical imaging, geophysical exploration and nondestructive testing. There are rich results in the literature for \eqref{eq:ipp1}, both theoretically and numerically. It is known that one has the unique identifiability for \eqref{eq:ipp1}, namely the correspondence between $u^\infty$ and $D$ is one-to-one. Many numerical reconstruction algorithms have been developed in solving \eqref{eq:ipp1} as well. In fact, it is impossible for us to list and discuss all of the existing numerical studies for \eqref{eq:ipp1} in the literature. In what follows, we only discuss a few selective ones to motivate our current study. Roughly speaking, the existing numerical approaches can be classified into two categories: quantitative ones and qualitative ones. Quantitative approaches employ Newton's linearization and/or optimization strategies to tackle \eqref{eq:ipp1} directly. We refer the reader to the Newton's iterative method \cite{CK19, HK04, Zhang2017}, the decomposition method \cite{CK19, Guo2011, Kirsch1988} and the recursive linearization method \cite{Bao2015, Bao2003} for the relevant results. The other class of approaches resorts to establishing a criterion to distinguish the interior and exterior of the obstacle $D$, and hence can qualitatively reconstruct the boundary of $D$. These include the linear sampling method \cite{Cakoni2006, Cakoni2011,CH05, LLZ2}, the factorization method \cite{Kirsch2007}, the direct sampling method \cite{GR, LLZ1, Liu2017} and the probe method \cite{NPS,Pot}.

In this article, we develop a novel quantitative reconstruction scheme for \eqref{eq:ipp1}. A key ingredient of our method is to connect the exterior scattering problem \eqref{eq:forwardpr} to its interior counterpart:
\begin{equation}\label{eq:inteNeu}
	\left\{
	\begin{array}{ll}
		\Delta v+k^2 v=0\quad &\text{in} \ D, \medskip \\
		\displaystyle \mathcal{B}(v)=0 \quad &\text{on} \ \partial D,
	\end{array}
	\right.
\end{equation}
where $v\in H^1(D)$. \eqref{eq:inteNeu} is the classical Dirichlet/Neumann Laplacian eigenvalue problem and $v$ represents an interior resonant mode. It has been well understood, say e.g. there exist infinitely many discrete eigenvalues accumulating only at $\infty$, and the eigenspace associated with each eigenvalue is finite-dimensional. It turns out that the the exterior problem \eqref{eq:forwardpr} and the interior problem \eqref{eq:inteNeu} are dual to each other, in particular in the sense that one can read off the spectral system of \eqref{eq:inteNeu} from the far-field data of \eqref{eq:forwardpr}. This is done by employing a certain indicating behaviour of the linear sampling method and by solving a constrained optimization problem, both only involving the far-field data in a simple manner, which form the first phase of the proposed method. In the second phase, we derive an iteration formula of Newton's type which can be used to quantitatively reconstruct the boundary of the obstacle. By utilizing the interior eigenfunctions determined in the first phase, the shape derivative involved in the Newton's iteration can be calculated in a very efficient way without the need to solve any forward scattering problem. To highlight the novel contributions of the current study, we present three remarks as follows.

First, determining the interior eigenvalues from the associated far-field data via the linear sampling method has been studied in \cite{CCH10,Lechleiter,LLWW,LiuSun}. In \cite{LLWW}, the determination of the interior Dirichlet Laplacian eigenfunctions has been further investigated. In fact, it is proposed in \cite{LLWW} that the Dirichlet eigenfunctions can be used to recover $\partial D$ (in a qualitative manner). This is natural since the Dirichlet eigenfunctions vanish on $\partial D$. However, such an idea cannot be extended to the sound-hard case since identifying the place where the Neumann data vanish requires knowing the normal vector of $\partial D$ which is equivalent to knowing $\partial D$. It is our aim of developing an imaging scheme that can recover sound-hard obstacles by making use of the interior resonant modes. It turns out that the scheme we develop can not only recover sound-hard obstacles but can also yield highly-accurate quantitative reconstructions. The key idea is to establish an iteration formula of Newton's type by using the homogeneous Neumann condition on $\partial D$, which can then be used for identifying the unknown $\partial D$. A salient feature of this Newton's iteration approach is that  it is essentially derivative-free in the sense that the involved shape derivatives can be explicitly calculated by using the Neumann eigenfunctions that have been already determined, and there is no need to calculate any direct scattering problem. Moreover, the idea can be directly extended to the sound-soft case to produce quantitative reconstructions that outperform the qualitative reconstructions in \cite{LLWW}. However, it can be seen even at this point that the sound-hard case is technically more challenging than the sound-soft case. Hence, we shall mainly stick our study to the sound-hard case and only remark the sound-soft case throughout the rest of the paper.

Second, the far-field data used in \eqref{eq:ipp1} is significantly overdetermined. In fact, most of the existing methods in the literature, in particular those qualitative ones, make use of far-field data corresponding to a fixed $k$. Moreover, it is widely conjectured that one can determine an obstacle by at most a few or even a single far-field measurements, namely a few $k$ and $d$ or even fixed $k$ and $d$ in \eqref{eq:ipp1}; see e.g. \cite{CDLZ,CY03,LPRX,LZ1,RL08} for related theoretical uniqueness and stability results and \cite{LLZ1} for related numerical reconstructions if a-priori geometrical knowledge is available on $D$. The overdetermination in \eqref{eq:ipp1}, especially on $k$, is mainly needed to compute the interior eigenvalues located inside $V$, which is the prerequisite to the eigenfunction determination. As mentioned in the first remark, the eigenfunctions are critical to avoid computing shape derivatives for the Newton's iteration in our method. Hence, the overdetermined data are an unobjectionable cost for achieving both high accuracy and high efficiency. It is interesting to note that all of the aforementioned qualitative methods inevitably involve shape derivatives and hence the calculation of a large amount of forward scattering problems, and moreover suffer from the local minima issue. In fact, to our best knowledge, no 3D quantitative reconstruction of a sound-hard obstacle was ever conducted in the literature due to the highly complicated computational nature. In Section 4, we present both 2D and 3D reconstructions and it can be seen that our method is computationally straightforward. Moreover, it can be seen that our method is robust and insensitive to the initial guess, and this is physically justifiable as the interior resonant modes carry the geometrical information of $D$ in a sensible way (though implicitly in the sound-hard case). It is also interesting to note the similarity shared by our method and the machine learning approaches for inverse obstacle problems. In \cite{GLWZ,YYL}, machine learning approaches were developed that can yield a highly-accurate reconstruction of a target obstacle by using only a few far-field measurements. However, a large amount of data as well as computations are needed to train the neural networks therein, and hence are computationally more costly.

Third, we would like to mention two practical scenarios where our method might find applications to corroborate our viewpoint in the above remark. In many practical applications, say e.g. photo-acoustic tomography, time-dependent data are collected which can yield multiple frequency data as needed in \eqref{eq:ipp1} via temporal Fourier transform \cite{Liu2015}. The other application is the bionic approach of generating human body shape \cite{LLTW}, where overdetermined data are not a practical drawback, but highly accurate reconstructions are needed. We shall consider the application of our method in those practical setups in our future work.

The rest of the paper is organized as follows. Section 2 introduces the linear sampling method to determine the interior eigenvalues from the multi-frequency far-field data. Then by the Herglotz wave approximation, we present an efficient optimization algorithm to reconstruct the interior eigenfunctions. In Section 3, a novel Newton iterative formula is proposed to reconstruct the sound-hard obstacle via the use of the interior  eigenfunctions. Numerical experiments are conducted in 4 to verify the promising features of our method.

\section{Determination of Neumann eigenvalues and eigenfunctions}

In this section, we aim to determine the interior eigenvalues and eigenfunctions to \eqref{eq:inteNeu} from the far-field data in \eqref{eq:ipp1} corresponding to the unknown $D$. As discussed in the previous section, we shall mainly focus on the sound-hard case, namely $\mathcal{B}(u)=\partial u/\partial\nu$ and only remark the extension to the sound-soft case.


To begin with, we introduce the linear sampling method for reconstructing the Neumann eigenvalues. The linear sampling method is a widely used method to recover the shape of the scatterer without {\it a priori} information of the boundary condition of the scatterer. The basic idea of the linear sampling method is choosing an approximate indicator function and distinguishing whether the sampling point is inside or outside the scatterer. To that end, we present the indicator function of the linear sampling method.
Define the test function by
\begin{equation*}
	\Phi^{\infty}(\hat{x},z,k)={\rm e}^{-\mathrm{i}k\hat{x}\cdot z}, \quad \hat{x}\in \mathbb{S}^{m-1},
\end{equation*}
where $z\in \mathbb{R}^m$ denotes the sampling point. The key ingredient of the linear sampling method is to find the kernel $g\in L^2(\mathbb{S}^{m-1})$ as a solution to the following integral equation
\begin{equation}\label{eq:ffeqn}
	(F_k g)(\hat{x})=\Phi^{\infty}(\hat{x},z,k),
\end{equation}
where $F_k$ is the far field operator from $L^2(\mathbb{S}^{m-1})$ to $L^2(\mathbb{S}^{m-1})$ and it is defined by
\begin{equation}\label{eq:ffope}
	(F_kg)(\hat{x}):=\int_{\mathbb{S}^{m-1}} u^{\infty}(\hat{x},d, k)\, g(d)\, \mathrm{d}s(d), \quad \hat{x}\in \mathbb{S}^{m-1}.
\end{equation}
We usually adopt  Tikhonov regularization  to solve the Fredholm integral equation \eqref{eq:ffeqn} since this equation is ill posed. Due to the existence of noisy data in practice, we suppose that $F_k^{\delta}$  is the corresponding operator to the noisy measurements  $u^{\infty,\delta}$ and then instead seek the unique minimizer $g_{z,k}^{\delta}\in L^2(\mathbb{S}^{m-1})$ of the following functional
\begin{equation}\label{eq:TikF}
	\|F_k^{\delta}g-\Phi^{\infty}(\hat{x},z,k)  \|^2_{ L^2(\mathbb{S}^{m-1})}+\epsilon \| g \|^2_{L^2(\mathbb{S}^{m-1})},
\end{equation}
where $\epsilon$ is a regularization parameter.

Next, we discuss how to recover the Neumann eigenvalues  by using the linear sampling method.  According to the rigorous justification of the rationale behind the linear sampling method (see \cite{Arens09, LiuSun}), one can use the following lemma to distinguish whether $k$ is a Neumann eigenvalue or not.

\begin{lem}\label{le:LinrSamg}
	Define the Herglotz wave function by
	\begin{equation*}
		H_k g(x):=\int_{\mathbb{S}^{m-1}} \mathrm{e}^{\mathrm{i}kx\cdot d} g(d) \, {\rm d}s(d).
	\end{equation*}
	Suppose $g_{z,k}^{\delta}$ be the unique minimizer of the Tikhonov functional \eqref{eq:TikF}. Then, for almost every $z\in D$, $\|H_k g_{z,k}^{\delta}\|_{H^1(D)}$ is bounded as $\delta\to 0$ if and only if $k$ is not a Neumann eigenvalue.
\end{lem}

\begin{rem}
	Due to the fact that the scatterer $D$ is unknown, it is impossible to identify the Neumann eigenvalues based on the behavior of $\|H_k g_{z,k}^{\delta}\|_{H^1(D)}$. Noting that $\|g_{z,k}^{\delta} \|_{L^2(\mathbb{S}^{m-1})}$ has the same behavior to $\|H_k g_{z,k}^{\delta}\|_{H^1(D)}$,  one can use the behavior of $\|g_{z,k}^{\delta} \|_{L^2(\mathbb{S}^{m-1})}$ to determine the Neumann eigenvalues.
\end{rem}

Next, we proceed to recover the corresponding Neumann eigenfunctions. Recall that the Herglotz wave function is defined by
\begin{equation}\label{eq:Herglotz}
	v_{g,k}(x):= H_k g(x) =\int_{\mathbb{S}^{m-1}} \mathrm{e}^{\mathrm{i}kx\cdot d} g(d) \,{\rm d}s(d), \quad x\in\mathbb{R}^m,
\end{equation}
where $g\in L^2(\mathbb{S}^{m-1})$ is called the Herglotz kernel of $v_{g,k}$. We first show that the Herglotz wave function can be used to approximate the solution of the Helmholtz equation by the following lemma.

\begin{lem}\label{le:HergDense}\cite{Weck04}
	Let $D \subset \mathbb{R}^m$ be a bounded domain of class $C^{\alpha, 1}$, $\alpha\in\mathbb{N}\cup\{0\}$ with a connected complement $\mathbb{R}^m\backslash\overline{D}$. Let $\mathbb{H}$ be the space of all Herglotz wave functions of the form \eqref{eq:Herglotz}. Define, respectively,
	\[
	\mathbb{H}(D):=\{ u|_{D}:u\in\mathbb{H} \},
	\]
	and
	\[
	\mathbb{U}(D):=\{u\in C^{\infty}(D): \Delta u +k^2 u=0 \text{ in } D\}.
	\]
	Then $\mathbb{H}(D)$ is dense in $\mathbb{U}(D)\cap H^{\alpha+1}(D)$ with respect to the $H^{\alpha+1}(D)$-norm.
\end{lem}

The following theorem plays an important role to determine the Neumann eigenfunctions.

\begin{thm}\label{th:Eigenf}
	Assume that $D$ is of class $C^{0,1}$ when $m=2$, and of class $C^{1,1}$ when $m=3$, and $\mathbb{R}^m\backslash\overline{D}$ is connected. Suppose that $k$ is a Neumann eigenvalue of $-\Delta$ in $D$ and $u_k$ is the corresponding eigenfunction. Let $\varepsilon \in \mathbb{R}_{+}$ be sufficiently small. Then there exists $g_{\varepsilon}\in L^2(\mathbb{S}^{m-1})$ such that
	\begin{equation}\label{eq:optimal}
		\|F_k g_{\varepsilon} \|_{L^2(\mathbb{S}^{m-1})}=\mathcal{O}(\varepsilon) \ \ \text{and} \ \   \|v_{g_{\varepsilon},k} \|_{H^1(D)}=\mathcal{O}(1),
	\end{equation}
	where $F_k$ is the far field operator defined by \eqref{eq:ffope} and $v_{g_{\varepsilon},k}$ is the Herglotz wave function defined by \eqref{eq:Herglotz} with the kernel $g_{\varepsilon}\in L^2(\mathbb{S}^{m-1})$.

	On the other hand, if we suppose that $k\in\mathbb{R}_+$ is a Neumann eigenvalue and $g_{\varepsilon}$ satisfies \eqref{eq:optimal}, then the Herglotz wave $v_{g_{\varepsilon},k}$ is an approximation to the Neumann eigenfunction $u_k$ associated with the Neumann eigenvalue $k$ in $H^1(D)$-norm.
\end{thm}

\begin{proof}
	Let $u_k$ be a normalized Neumann eigenfunction associated with the Neumann eigenvalue $k$,
	then $u_k\in H^1(D)$ is a solution to the Neumann eigenvalue problem
	\begin{equation*}
		\Delta u_k +k^2 u_k=0 \ \ \text{in} \ D, \quad
		\frac{\partial u_k}{\partial \nu}=0 \ \  \text{on} \  \partial D.
	\end{equation*}
	According to the denseness in Lemma \ref{le:HergDense}, for any sufficiently small $\varepsilon>0$, there exists $g_{\varepsilon}\in L^2(\mathbb{S}^{m-1})$ such that
	\begin{equation*}
		\| v_{g_{\varepsilon},k} -u_k \|_{H^1(D)}< \varepsilon.
	\end{equation*}
	where $v_{g_{\varepsilon},k}$ is the Herglotz wave function  with the kernel $g_{\varepsilon}$.
	By the triangle inequality, one can find that
	\begin{equation*}
		\| v_{g_{\varepsilon},k} \|_{H^1(D)}\leq \|v_{g_{\varepsilon},k}  -u_k \|_{H^1(D)}+\| u_k \|_{H^1(D)}
		<\varepsilon +\| u_k \|_{H^1(D)}.
	\end{equation*}
	Similarly, we have
	\begin{equation*}
		\| v_{g_{\varepsilon},k}  \|_{H^1(D)}\geq \| u_k \|_{H^1(D)}-\| v_{g_{\varepsilon},k}  -u_k \|_{H^1(D)}
		>\| u_k \|_{H^1(D)}-\varepsilon.
	\end{equation*}
	Since $u_k\in H^1(D)$ is a normalized eigenfunction, using the last two equations, one can obtain that
	\begin{equation*}
		\| v_{g_{\varepsilon},k}  \|_{H^1(D)}=\mathcal{O}(1).
	\end{equation*}
	In the following, we let $ a \lesssim  b$ stand for $a\leq C b$, where $C>0$ is a generic constant.  By the trace theorem, we can derive that
	\begin{equation}\label{eq:dd1}
		\left \|\frac{\partial v_{g_{\varepsilon},k} }{\partial \nu} -\frac{\partial u_k}{\partial \nu}\right\|_{H^{-1/2}(\partial D)}\leq \varepsilon.
	\end{equation}
	Noting that $\partial u_k/ \partial \nu=0$ on $\partial D$, we clearly have from \eqref{eq:dd1} that
	\begin{equation}\label{eq:numann}
		\left \|\frac{\partial v_{g_{\varepsilon},k}  }{\partial \nu} \right\|_{H^{1/2}(\partial D)}
		\lesssim \varepsilon.
	\end{equation}
	It is directly verified that $F_k g_{\varepsilon}$ is the far-field pattern of the exterior scattering problem \eqref{eq:forwardpr} associated with the incident field $v_{g_{\varepsilon},k} $. Hence, based on the well-posedness of the forward scattering problem \eqref{eq:forwardpr}, together with \eqref{eq:numann}, one can conclude that
	\begin{equation}\label{eq:dd2}
		\|F_k g_{\varepsilon} \|_{L^2(\mathbb{S}^{m-1})}=\mathcal{O}(\varepsilon),
	\end{equation}
	which proves the first part of the theorem.

	Next, we prove that the Herglotz wave $v_{g_{\varepsilon},k}$ is an approximation to the Neumann eigenfunction $u_k$.  Let $u^i=v_{g_{\varepsilon},k}$, $u^s_{g_{\varepsilon},k}$ and $u^t_{g_{\varepsilon},k}$ be, respectively, the incident, scattered, and total wave fields. It is clear that one has
	\begin{equation}\label{eq:scaSys}
		\left\{
		\begin{array}{ll}
			\Delta v_{g_{\varepsilon},k}+k^2 v_{g_{\varepsilon},k}=0\quad &\text{in} \ D, \medskip \\
			\displaystyle \frac{\partial v_{g_{\varepsilon},k}}{\partial \nu}=-\frac{\partial u^s_{g_{\varepsilon},k}}{\partial \nu} \quad &\text{on} \ \partial D.
		\end{array}
		\right.
	\end{equation}
	
	As we mentioned before, $F_kg_\varepsilon$ is the far-field pattern of $u_{g_\varepsilon, k}^s$. According to \eqref{eq:optimal} and the quantitative Rellich theorem (cf. \cite{Blasten16} as well as Remark~\ref{rem:nn1} in what follows), there is
	\begin{equation}\label{eq:ssn2}
		\left\| \frac{\partial u^s_{g_{\varepsilon},k}}{\partial \nu} \right\|_{L^2(\partial D)}\leq \psi (\varepsilon),
	\end{equation}
	where $\psi$ is a real-valued function of logarithmic type and satisfies $\psi (\varepsilon)\to 0$ as $\varepsilon\to +0$.
	Let $u_k$ be a Neumann eigenfunction associated with $k$, namely,
	\begin{equation}\label{eq:NeuSys}
		\left\{
		\begin{array}{ll}
			\Delta u_{k}+k^2 u_{k}=0\quad &\text{in} \ D, \medskip \\
			\displaystyle \frac{\partial u_{k}}{\partial \nu}=0 \quad &\text{on} \ \partial D.
		\end{array}
		\right.
	\end{equation}
	With the help of the Fredholm theory for elliptic boundary value problems (see e.g. Theorem 4.10 in \cite{Mclean}), the solution set of the system \eqref{eq:scaSys} is given by $v_{g_{\varepsilon},k}+\mathbb{V}$ with $\mathbb{V}$ being the finite-dimensional eigen-space to \eqref{eq:NeuSys}. Moreover,
	\begin{equation*}
		\|v_{g_{\varepsilon},k}+ \mathbb{V} \|_{H^1(D)/\mathbb{V}}\leq C\left\| \frac{\partial u^s_{g_{\varepsilon},k}}{\partial \nu} \right\|_{H^{-1/2}(\partial D)}\leq C\psi (\varepsilon),
	\end{equation*}
	which clearly indicates that $v_{g_\varepsilon, k}$ is indeed an approximation to a Neumann eigenfunction associated with $k$.
	
	The proof is complete.
\end{proof}

\begin{rem}\label{rem:nn1}
	
	In the proof of Theorem~\ref{th:Eigenf}, we make use of the so-called quantitative Rellich theorem which states that if the far-field is smaller than $\varepsilon$ (i.e. $\|F_kg_{\varepsilon}\|_{L^2(\mathbb{S}^{m-1})}\leq \varepsilon$ in our case), then the scattered field (i.e. $u_{g_\varepsilon, k}^s$ in our case) is also small up to the boundary of the scatterer in the sense of \eqref{eq:ssn2}, where $\psi(\varepsilon)$ is the stability function in \cite{Blasten16}. In fact, in \cite{Blasten16}, the quantitative Rellich theorem is established for medium scattering. But as long as the scattered field is H\"older continuous up the boundary of the scatterer, the result in \cite{Blasten16} can be straightforwardly extended to the case of obstacle scattering. In Theorem~\ref{th:Eigenf}, the regularity assumptions on $\partial D$ guarantees that the scattered field is indeed H\"older continuous up to $\partial D$. In fact, let us consider \eqref{eq:forwardpr} to ease the exposition. In two dimensions, $\partial u/\partial\nu|_{\partial D}\in L^2(\partial D)$. It follows from regularity results for elliptic problems in Lipschitz domains \cite[Theorem~4.24]{Mclean} that $u|_{\partial D}\in H^1(\partial D)$, and hence from \cite[Theorem~6.12]{Mclean} and the accompanying discussion that $u\in H_{loc}^{3/2}(\mathbb{R}^2\backslash\overline{D})$. Then by the standard Sobolev embedding theorem, one can directly verify that $u$ is H\"older continuous up to $\partial D$ and therefore $u^s=u-u^i$ is also H\"older continuous up to $\partial D$. In three dimensions, since $\partial D\in C^{1,1}$, one can make use of the \cite[Theorem~4.18]{Mclean} to show that $u\in H_{loc}^2(\mathbb{R}^3\backslash\overline{D})$. Again by the Sobolev embedding theorem, one can see that $u$ and hence $u^s$ are H\"older continuous up to $\partial D$. We believe the regularity assumption in three dimensions can be relaxed to be purely Lipschitz as that in two dimensions. However, this is not the focus of the current study. We also refer to \cite{RL08} for a different approach of quantitatively continuing the far-field data to the boundary of the scatterer.
	

\end{rem}

By Theorem \ref{th:Eigenf} and normalization if necessary, we can say that the following optimization problem:
\begin{equation}\label{eq:AlterOptimization0}
	\min\limits_{g\in L^2(\mathbb{S}^{m-1})} \|F_k g \|_{L^2(\mathbb{S}^{m-1})} \quad \text{    s.t.  } \|v_{g,k} \|_{L^2(D)}=\mathcal{O}(1)
\end{equation}
exists at least one solution $g \in L^2(\mathbb{S}^{m-1})$ when $k$ is a Neumann eigenvalue in $D$.
Since $D$ is unknown, it is unpractical to solve the optimization problem \eqref{eq:AlterOptimization0} with the constraint term $\|v_{g,k} \|_{L^2(D)}=\mathcal{O}(1)$.
Thus, we consider the following optimization problem instead:
\begin{equation}\label{eq:AlterOptimization}
	\min\limits_{g\in L^2(\mathbb{S}^{m-1})} \|F_k g \|_{L^2(\mathbb{S}^{m-1})} \quad \text{    s.t.  } \|v_{g,k} \|_{L^2(B)}=\mathcal{O}(1),
\end{equation}
where $B$ is bounded domain such that $D\subset B$. In practice, one can simply choose $B$ to be a large ball containing $D$. In summarizing our discussion above, we can formulate the following scheme ({\bf Scheme I}) to determine the Neumann eigenvalues and the corresponding eigenfunctions. It is emphasized that all the results in this section hold equally for the sound-soft case.
\begin{table}[htbp]
	\begin{tabular}{p{1.5cm} p{11cm}}
		\toprule
		\multicolumn{2}{l}{\textbf{Scheme \uppercase\expandafter{\romannumeral1}:}\ Determination of Neumann eigenvalues and eigenfunctions}\\
		\midrule
		Step 1 & Collect a family of far-field data $u^{\infty,\delta}(\hat{x},d,k)$ for $(\hat{x},d,k)\in \mathbb{S}^{m-1}\times \mathbb{S}^{m-1} \times V$, where $V$ is an open interval in $\mathbb{R}_{+}$.\\
		Step 2 & Pick a point $z \in D$ (a priori information) and for each $k\in V$, find the minimizer $g_{z,k}^{\delta}$ of \eqref{eq:TikF}.\\ 
		Step 3 & Plot $\| g_{z,k}^{\delta}\|_{L^2(\mathbb{S}^{m-1})}$ against $k\in V$ and find the Neumann eigenvalues from the peaks of the graph.\\
		Step 4 & For each determined Neumann eigenvalue, solve the optimization problem \eqref{eq:AlterOptimization} and obtain the Herglotz kernel function $g_k$ by using the gradient total least square method as proposed in \cite{LLWW}.\\
		Step 5 & With the computed Herglotz kernel function $g_k$, then the corresponding eigenfunction $u_k$ can be approximated by the Herglotz wave function according to the definition \eqref{eq:Herglotz}.\\
		\bottomrule
	\end{tabular}
\end{table}

\section{Newton-type method based on the interior resonant modes}
In this section, we develop a novel Newton-type method to reconstruct the shape of a sound-hard obstacle based on the interior resonant modes determined in the previous section.

Let $\gamma$ be a closed curve (in $\mathbb{R}^2$) or a closed surface (in $\mathbb{R}^3$), which is parameterized by
\begin{equation*}
	\begin{cases}
		& \gamma= \{z(\phi):  \phi\in[0,\,2\pi]\}, \qquad \qquad \qquad \quad \, m=2,\medskip\\
		& \gamma= \{z(\theta, \phi): (\theta,\phi)\in[0,\, \pi]\times[0,\,2\pi]\}, \quad m=3.
	\end{cases}
\end{equation*}
Here, $\gamma$ signifies the boundary of a bounded domain which shall be involved in the Newton's iteration in what follows.
For a fixed total field $u$, we define the operator $G$ that  maps the boundary contour $\gamma$ onto the trace of $\partial u/\partial \nu$ on $\gamma$, that is,
\begin{equation}\label{eq:mm1}
	G: \gamma \mapsto \frac{\partial u}{\partial \nu}.
\end{equation}
In terms of the above operator $G$,  we seek the parameterization $z$ of the boundary contour $\gamma$ such that the total field $u$ satisfies the Neumann boundary condition, i.e.,
\begin{equation}\label{eq:Neumann}
	G(z)=0,  \quad z\in \gamma.
\end{equation}

Next, we introduce the Newton-type method to recover the boundary of the obstacle. Let $\gamma_n$ be the approximation to the boundary $\partial D$ at the $n$-iteration, $n=0,1,2,\cdots$. Following the idea of Newton method, we replace the previous nonlinear equation \eqref{eq:Neumann} by the linearized equation
\begin{equation}\label{eq:linear}
	G(z_n)+ G'(z_n)h=0, \quad z_n\in \gamma_n,
\end{equation}
where $h$ denotes the shift.
Actually,  the key point for solving the last linearized equation is to determine the Fr\'{e}chet derivative of the operator $G$. Inspired by the work \cite{Kress2007}, we seek to improve and  present a simpler  Fr\'{e}chet derivative in two and three dimensions, respectively.

For the two-dimensional case, we assume that $z=(z^{(1)}, z^{(2)})^{\top}$. Then the exterior normal vector can be represented by
\begin{equation*}
	\nu:=\frac{z_{\phi}^{\bot}}{|z_{\phi}^{\bot}|},
\end{equation*}
where
$ z_{\phi}^{\bot}=\left( {\partial z^{(2)}}/{\partial \phi},\, -{\partial z^{(1)}}/{\partial \phi} \right)^{\top}$.  Now we characterize the Fr\'echet derivative in the following theorem.

\begin{thm}\label{thm:2D}
	The operator $G: C^2([0,2\pi]) \rightarrow C([0,2\pi])$  is  Fr\'{e}chet differentiable and its derivative is given by
	\begin{equation*}
		G'(z)h=\frac{1}{|z_{\phi}|}(I-\nu\nu^{\top})h_{\phi}^{\bot}\cdot \nabla u +\nu \cdot ( \nabla \nabla^{\top} u \, h),
	\end{equation*}
	where $I$ is the $2\times 2$ identity matrix and  $h_{\phi}^{\bot}=(\partial h^{(2)}/ \partial \phi,  -\partial h^{(1)}/ \partial \phi)^\top$.
\end{thm}

\begin{proof}
	Recall that
	\begin{equation*}
		G(z)=\nu(z) \cdot \nabla u(z),
	\end{equation*}
	then we have the following decomposition
	\begin{equation}\label{eq:decomposition}
		G(z+h)-G(z)=(\nu(z+h)-\nu(z))\cdot \nabla u(z)+ \nu(z)\cdot (\nabla u(z+h)-\nabla u(z)).
	\end{equation}
	Noting that $\nu={z_{\phi}^{\bot}}/{|z_{\phi}^{\bot}|}$, using Taylor's formula, one can deduce that
	\begin{equation*}
		\begin{aligned}
			\nu(z+h)-\nu(z)&=\frac{z_{\phi}^{\bot}+h_{\phi}^{\bot}}{|z_{\phi}^{\bot}+h_{\phi}^{\bot}|}-\frac{z_{\phi}^{\bot}}{|z_{\phi}^{\bot}|}\\
			&=\frac{\partial } {\partial z_{\phi}^{\bot}} \left(\frac{z_{\phi}^{\bot}}{|z_{\phi}^{\bot}|}\right) \, h_{\phi}^{\bot}+\mathcal{O}(|h_{\phi}^{\bot}|^2)\\
			&=\frac{1}{|z_{\phi}^{\bot}|}(I-\nu\nu^{\top})h_{\phi}^{\bot}+\mathcal{O}(|h_{\phi}^{\bot}|^2).
		\end{aligned}
	\end{equation*}
	Similarly, we can derive that
	\begin{equation*}
		\begin{aligned}
			\nabla u(z+h)-\nabla u(z)
			&=\frac{\partial }{\partial z}(\nabla u(z))\, h+ \mathcal{O}(|h|^2)\\
			&=\nabla \nabla^{\top} u(z)\, h+ \mathcal{O}(|h|^2).
		\end{aligned}
	\end{equation*}
	Substituting the last two equations into \eqref{eq:decomposition} and using the definition of the Fr\'echet derivative,  one can deduce that
	\begin{equation*}
		G'(z)h=\frac{1}{|z_{\phi}|}(I-\nu\nu^{\top})h_{\phi}^{\bot}\cdot \nabla u +\nu \cdot ( \nabla \nabla^{\top} u \, h).
	\end{equation*}
\end{proof}

For the three-dimensional case, we define two orthogonal unit tangential vector fields $\tau_1$ and $\tau_2$ on the surface $\gamma$, that is,
\begin{equation*}
	\tau_1=\frac{z_{\theta}}{|z_{\theta}|}, \quad \tau_2=\frac{z_{\phi}-(z_{\phi} \cdot z_{\theta})z_{\theta}}{|z_{\phi}-(z_{\phi} \cdot z_{\theta})z_{\theta}|},
\end{equation*}
where $z_{\theta}=\partial z/\partial \theta$ and $z_{\phi}=\partial z/\partial \phi$.
Therefore, the normal vector can be defined by
\begin{equation*}
	\nu:=\frac{\tau_1\times \tau_2}{|\tau_1\times \tau_2|}
	=\frac{z_{\theta}\times z_{\phi}}{|z_{\theta}\times z_{\phi}|}.
\end{equation*}
In the following theorem, we characterize the Fr\'echet derivative of $G$ in three dimension.

\begin{thm}\label{thm:3D}
	The operator $G: C^2([0,\pi]\times[0,2\pi]) \rightarrow C([0,\pi]\times[0,2\pi])$  is  Fr\'{e}chet differentiable and
	its derivative is given by
	\begin{equation*}
		\begin{aligned}
			G'(z)h&=\frac{1}{|z_{\theta}\times z_{\phi}|}\left((I-\nu\nu^{\top})(h_{\theta}\times z_{\phi}+z_{\theta}\times h_{\phi}) \right)\cdot \nabla u +\nu \cdot ( \nabla \nabla^{\top} u \, h),
		\end{aligned}
	\end{equation*}
	where $I$ is the $3\times 3$ identity matrix,  $h_{\theta}=\partial h/ \partial \theta$ and $h_{\phi}=\partial h/ \partial \phi$.
\end{thm}

\begin{proof}

	Using Taylor's formula, one can derive that
	\begin{equation*}
		\begin{aligned}
			\nu(z+h)-\nu(z)
			&=\frac{(z+h)_{\theta}\times (z+h)_{\phi}}{|(z+h)_{\theta}\times (z+h)_{\phi}|}-\frac{z_{\theta}\times z_{\phi}}{|z_{\theta}\times z_{\phi}|}\\
			&=\frac{\partial}{\partial z_{\theta}}\left(\frac{z_{\theta}\times z_{\phi}}{|z_{\theta}\times z_{\phi}|}\right) h_{\theta}+\frac{\partial}{\partial z_{\phi}}\left(\frac{z_{\theta}\times z_{\phi}}{|z_{\theta}\times z_{\phi}|}\right) h_{\phi}+\mathcal{O}(|h_{\theta}|^2)+\mathcal{O}(|h_{\phi}|^2)\\
			&=\frac{1}{|z_{\theta}\times z_{\phi}|}\left[h_{\theta}\times z_{\phi}+z_{\theta}\times h_{\phi}-\nu(  (z_{\phi}\times \nu)\cdot h_{\theta}+ (\nu \times z_{\theta} )\cdot h_{\phi} ) \right]\\
			&\quad +\mathcal{O}(|h_{\theta}|^2)+\mathcal{O}(|h_{\phi}|^2).
		\end{aligned}
	\end{equation*}
	According to the mixed product, we have
	\begin{equation*}
		(z_{\phi}\times \nu)\cdot h_{\theta}=\nu \cdot (h_{\theta}\times z_{\phi}), \quad
		(\nu \times z_{\theta} )\cdot h_{\phi} =\nu \cdot (z_{\theta}\times h_{\phi}).
	\end{equation*}
	Similar to the two dimensional case, by a straight forward calculation, one can deduce that
	\begin{equation*}
		\begin{aligned}
			G'(z)h&=\frac{1}{|z_{\theta}\times z_{\phi}|}\left((I-\nu\nu^{\top})(h_{\theta}\times z_{\phi}+z_{\theta}\times h_{\phi}) \right)\cdot \nabla u +\nu \cdot ( \nabla \nabla^{\top} u \, h).
		\end{aligned}
	\end{equation*}
	
\end{proof}

To avoid calculating the direct scattering at each iteration, we use the eigenfunction $v_{g,k}$ to approximate the wave field $u$, namely,
\begin{equation*}
	u(x)\approx v_{g,k}(x)=\int_{\mathbb{S}^{m-1}} \mathrm{e}^{\mathrm{i}kx\cdot d} g(d) \,{\rm d}s(d), \quad x\in\mathbb{R}^m.
\end{equation*}
Then the gradient and the Hessian Matrix of $u$ defined in Theorem \ref{thm:2D} and \ref{thm:3D} can be represented by
\begin{equation*}
	\begin{aligned}
		&\nabla u= \mathrm{i}k  \int_{\mathbb{S}^{m-1}} d\, \mathrm{e}^{\mathrm{i}kx\cdot d} g(d) \,{\rm d}s(d), \\
		&\nabla \nabla^{\top} u= -k^2  \int_{\mathbb{S}^{m-1}} dd^{\top}\,\mathrm{e}^{\mathrm{i}kx\cdot d} g(d) \,{\rm d}s(d),
	\end{aligned}
\end{equation*}
where the Herglotz kernel $g$ is determined by solving the optimization problem \eqref{eq:AlterOptimization}.

\begin{rem}
	It is noted that the Fr\'echet derivatives presented in Theorem \ref{thm:2D} and \ref{thm:3D} are more simply and  intuitionistic compared with the formula defined in \cite{Kress2007}. Moreover,  based on the Herglotz wave approximation, it is very easy to numerically calculate the Fr\'echet derivatives since only cheap integrations are involved in the evaluation of the gradient and Hessian Matrix of $u$.
\end{rem}

We would like to emphasize that it is non-trivial to show that $G'$ is injective, namely, the operator $G'$ may be not invertible. Therefore, we employ the standard Tikhonov regularization scheme with a regularization parameter $\alpha$ for solving  linearized equation \eqref{eq:linear} at each iteration. As discussed above, we propose the following  Newton-type scheme.

\begin{alg}\label{eq:alg1}(Newton-type method)
	Assume that $\alpha>0$ is a regularization parameter. Given an
	initial parameter $z_0$, then the approximation $z_n$ is computed by
	\begin{equation*}
		z_{n}=z_{n-1}-\left(\alpha I+ (G'(z_{n-1}, k))^*\, G'(z_{n-1}, k)\right)^{-1} (G'(z_{n-1}, k))^*\,G(z_{n-1}, k),
	\end{equation*}
	where $k$ is  a Neumann eigenvalue.
\end{alg}

Noting that the Newton iteration method usually produces local minima for solving  the inverse obstacle problems.  To overcome local minimum and extend the convergence range, we propose a multifrequency Newton-type iterative algorithm.
\begin{alg}\label{eq:alg2}(Multifrequency Newton-type method)
	Assume that $\alpha>0$ is a regularization parameter. Given an
	initial parameter $z_0$, then the approximation $z_n$ is computed by
	\begin{equation*}
		z_{n}=z_{n-1}-\left(\alpha I+ (F'(z_{n-1}))^*\, F'(z_{n-1})\right)^{-1} (F'(z_{n-1}))^*\,F(z_{n-1}),
	\end{equation*}
	where $F(z_{n-1})=[G(z_{n-1}, k_1), G(z_{n-1}, k_2),\cdots, G(z_{n-1}, k_{\ell})]^{\top}$ and  $k_1, k_2, \cdots, k_{\ell}$ are $\ell$ different  Neumann eigenvalues.
\end{alg}

Finally, we would like to remark the extension to the sound-soft case. Indeed, by modifying the definition of the operator $G$ in \eqref{eq:mm1} via replacing $\partial u/\partial \nu$ by $u$, all the results in this section can be readily extended to the sound-soft case.

\section{Numerical experiments}
In this section, several numerical experiments are presented to verify the effectiveness and efficiency of the proposed methods. All the numerical results are conducted for recovering sound-hard obstacles, which present more challenges than the sound-soft case. To avoid inverse crime, the artificial far-field data are calculated by the finite element method, which is written as
\begin{equation*}
	\{ u^{\infty}(\hat{x}_i, d_j; k_s),\quad   i=1,2,\cdots, M, \ j=1,2,\cdots,N, \ s=1,2,\cdots,L. \}
\end{equation*}
Here $\hat{x}_i$ denotes the observation direction, $d_j$ denotes the incident direction and $k_s$ denotes the wavenumber. The observation and incident directions are chosen as the
equidistantly distributed points on the unit circle (2D) or unit sphere (3D). Moreover, the
wavenumbers are also equidistantly distributed in the open interval $V\subset \mathbb{R}_+$.
To test the stability of the method, for any fixed wavenumber, we add some random noise to the measurement matrix:
\begin{equation*}
	U^{\delta}=U+\delta \|U\|\frac{R_1+R_2 \mathrm{i}}{\|R_1+R_2 \mathrm{i}\|}, \quad U(i,j)=u^{\infty}(\hat{x}_i, d_j),
\end{equation*}
where $\delta>0$ is the noise level, and  $R_1$ and $R_2$ are two uniform random matrices that ranging from $-1$ to $1$.

\subsection{Recover the eigenvalues and eigenfunctions}
\begin{figure}
	\centering
	\subfigure[no noise]{\includegraphics[width=0.9\textwidth]
		{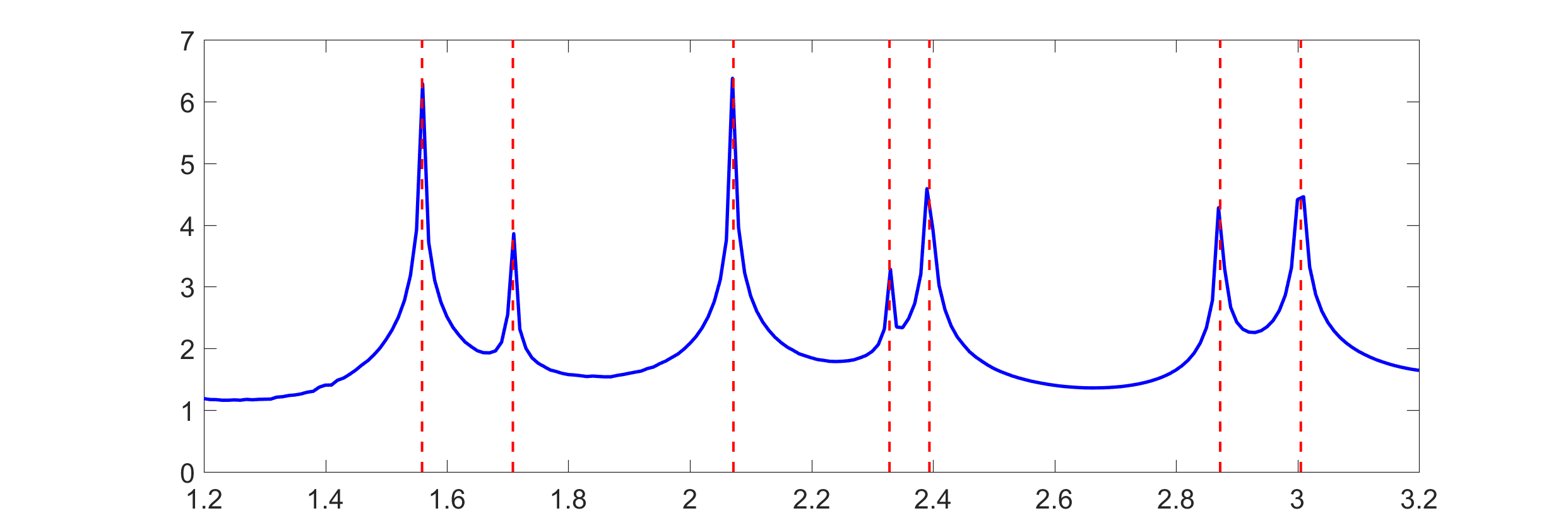}}\\
	\subfigure[1\% noise]{\includegraphics[width=0.9\textwidth]
		{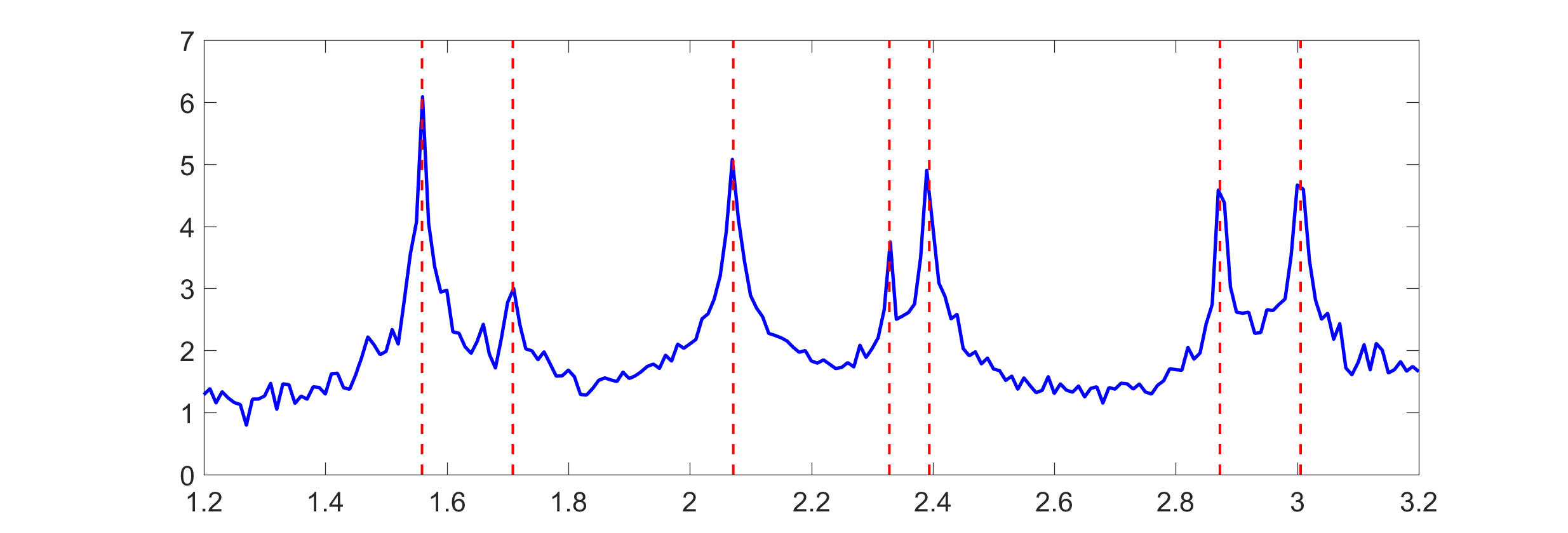}}
	\caption{\label{fig:eigenvalue} Plots of the indicator function with different noise level.
		The solid blue line denotes the value of $\|g_{z, k}^{\delta}\|_{L^2(\mathbb{S}^1)}$
		for $z=(1,\,1)$ against $k \in [1.2, \,  3.2]$; the dashed red lines denote the location of eigenvalues computed by the finite element method.
	}.
\end{figure}

In this part, we shall devote to recovering the eigenvalues and the corresponding eigenfunctions from the measured far field data. Here we consider a pear-shaped domain in two dimension, which is parameterized as
\begin{equation}\label{eq:pear}
	x(\phi)=(2+0.3\cos 3\phi)(\cos \phi, \ \sin \phi), \quad \phi\in[0,\, 2\pi].
\end{equation}
We set $M=64$, $N=64$ and $L=2000$, that is, the artificial far-field data are obtain at $64$ observation directions, $64$ incident directions and $2000$ equally distributed wavenumbers in the interval $[1.2, \, 3.2]$.

To begin with, we introduce the linear sampling method to determine the eigenvalues.
According to \textbf {Scheme I}, we plot the indicator function $\|g_{z, k}^{\delta}\|_{L^2(\mathbb{S}^1)}$ for $k\in [1.2, 3.2]$ in figure \ref{fig:eigenvalue}, where the interior test point is given by $z=(1,1)$.
In figure \ref{fig:eigenvalue} (a), the dashed red lines denote the location of eigenvalues computed by the finite element method and the solid blue line denotes the value of $\|g_{z, k}^{\delta}\|_{L^2(\mathbb{S}^1)}$.
As expected, one can observe that the indicator function (solid blue line) has clear spikes near the locations of the real eigenvalues (dashed red lines).
Thus, we can pick up the eigenvalues via the locations of the spikes.
Moreover,  we present the plot of the indicator function $\|g_{z, k}^{\delta}\|_{L^2(\mathbb{S}^1)}$ with $1\%$ noise in figure \ref{fig:eigenvalue} (b). Although the indicator function exhibits oscillating phenomenon, the eigenvalues can still be picked up from the locations of the spikes. Thus, the recovered Neumann eigenvalues are robust to the noise. In order to  exhibit the accuracy of the recovery results quantitatively, we also present the eigenvalues computed by the finite element method (FEM) in table \ref{tab1:eigenvalues}.
One can observe that the the linear sampling method (LSM) is valid to pick up  the eigenvalues.

\begin{table}[h]
	\center
	\caption{ The first seven real eigenvalues of the square domain. FEM:  computed the exact domain by the  finite element method; LSM: recovered from the far-field data by the linear sampling method. }\label{tab1:eigenvalues}
	\begin{tabular}{lccccccccc}
		\toprule
		Index of eigenvalue  & $1$     & $2$     &$3$   &$4$   &$5$    &$6$   &$7$  \\
		\midrule
		FEM:                &$1.559$   &$1.709$   &$2.072$  &$2.329$  &$2.394$    &$2.873$  & 3.006 \\
		LSM with no noise : &$1.559$   &$1.708$   &$2.071$  &$2.329$  &$2.394$    &$2.873$  & 3.005  \\
		LSM with 1\% nosie: &$1.562$   &$1.708$   &$2.073$  &$2.328$  &$2.395$    &$2.873$  & 3.004   \\
		\bottomrule
	\end{tabular}
\end{table}

\begin{figure}
	\centering
	\subfigure[$k_1=1.559$]{\includegraphics[width=0.32\textwidth]
		{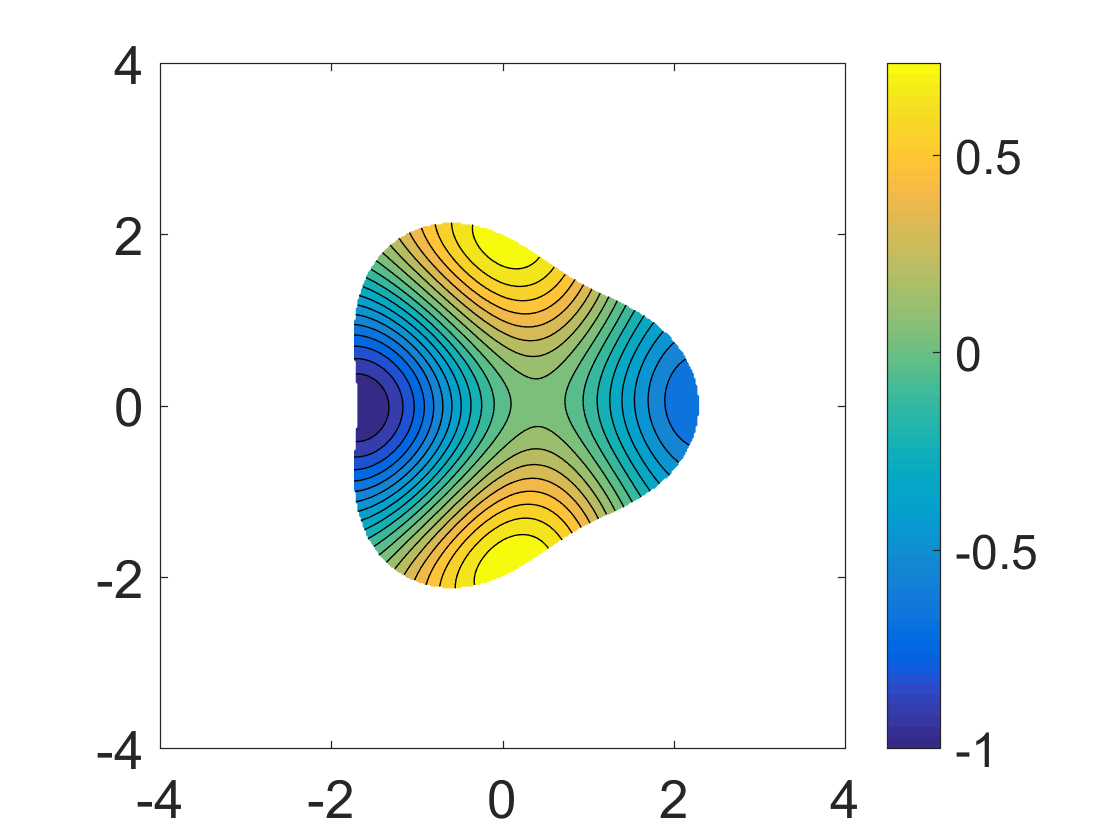}}
	\subfigure[$k_2=1.709$]{\includegraphics[width=0.32\textwidth]
		{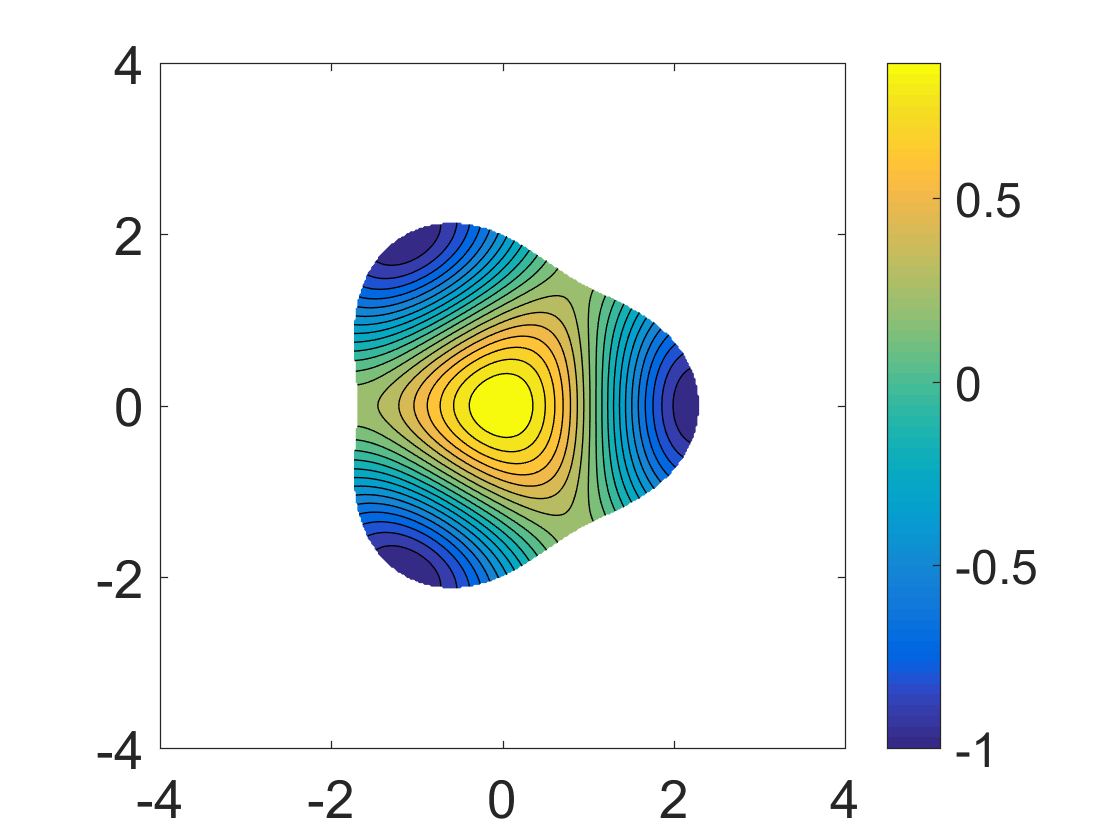}}
	\subfigure[$k_3=2.072$]{\includegraphics[width=0.32\textwidth]
		{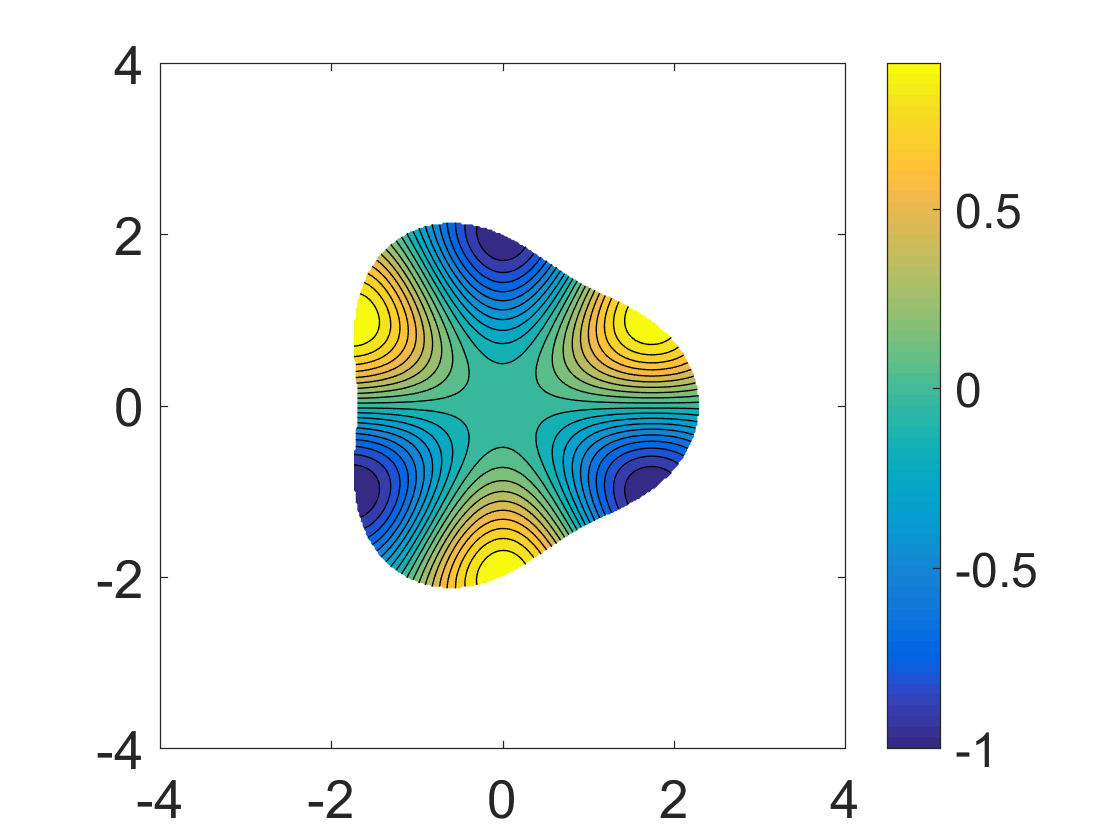}}\\
	\subfigure[$k_1=1.562$]{\includegraphics[width=0.32\textwidth]
		{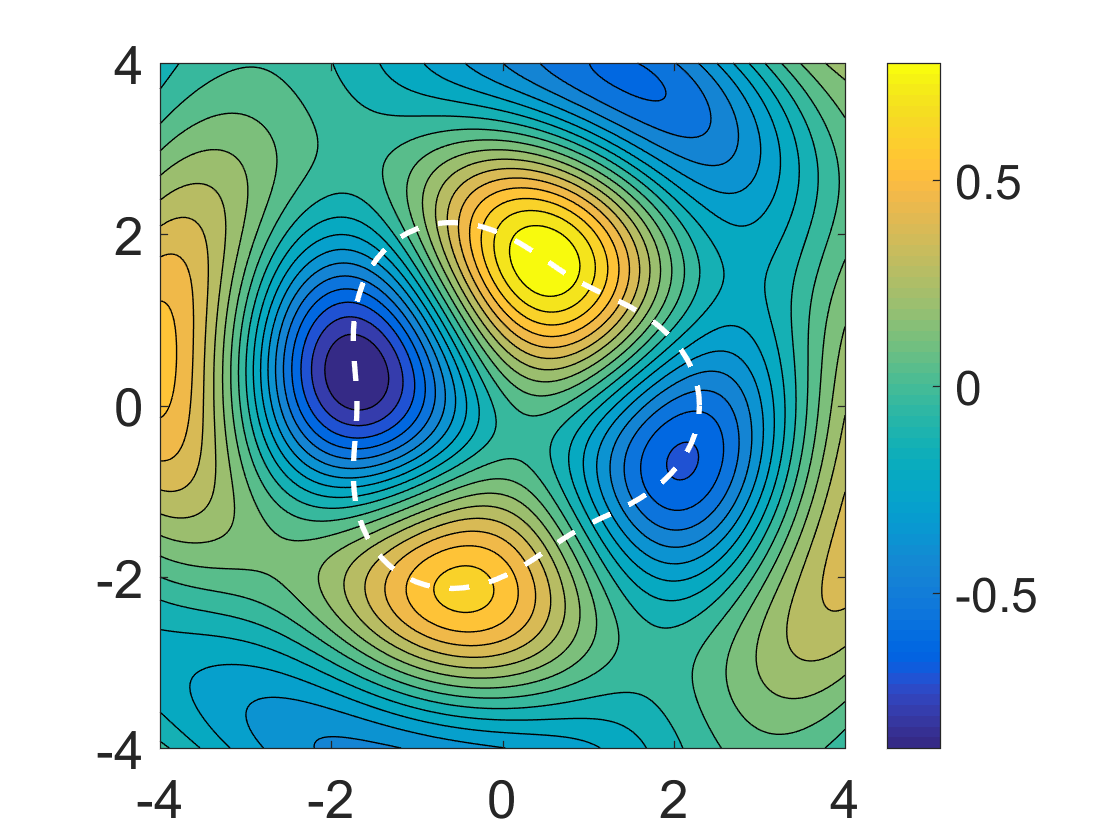}}
	\subfigure[$k_2=1.708$]{\includegraphics[width=0.32\textwidth]
		{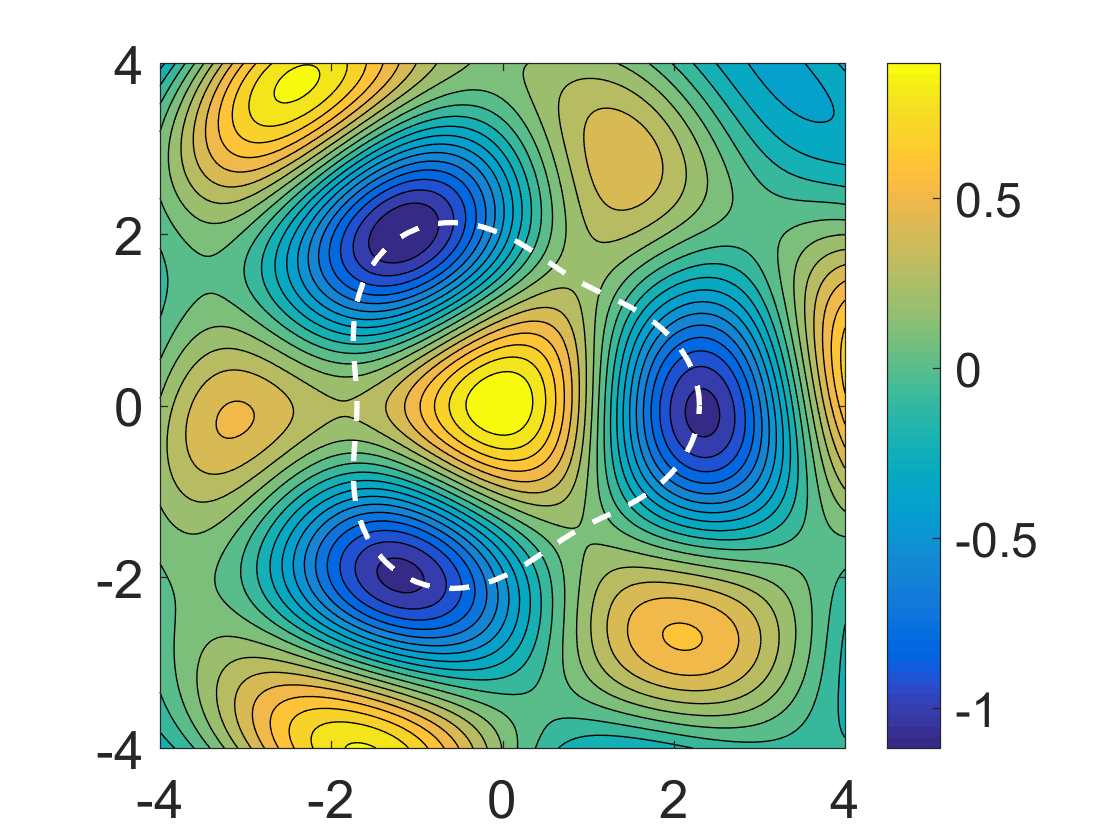}}
	\subfigure[$k_3=2.073$]{\includegraphics[width=0.32\textwidth]
		{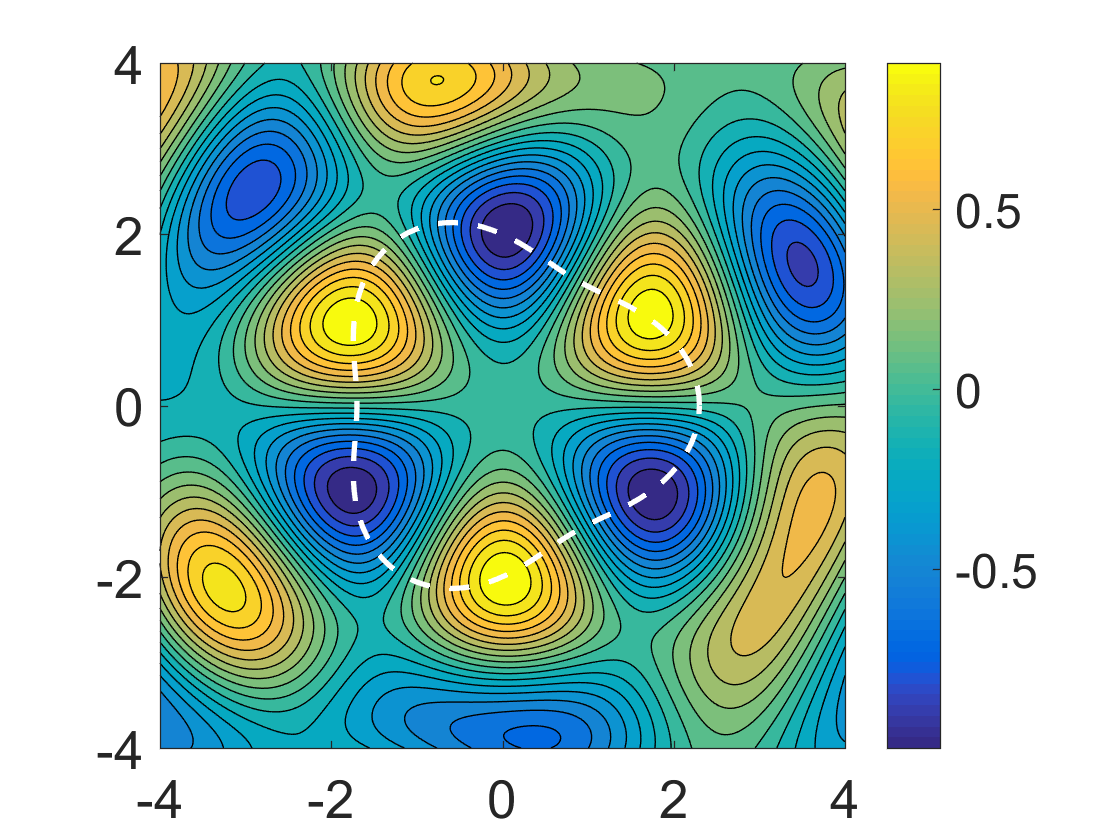}}
	\caption{\label{fig:eigenfunction} Contour plots of the exact and reconstructed eigenfunctions with different eigenvalues. The top row: the eigenfunctions recovered by using the finite element method;  the bottom row: the Herglotz wave recovered by using the gradient total least square method.
	}.
\end{figure}

Next, we aim to recover the corresponding  eigenfunctions from the far field data associated with Neumann eigenvalues.
Here, we shall use the gradient total least square method as proposed in \cite{LLWW}.
To reduce the oscillations, we add a penalty term to the optimization problem \eqref{eq:AlterOptimization}, that is,
\begin{equation*}
	\min\limits_{g\in L^2(\mathbb{S}^{m-1})} \|F_k g \|_{L^2(\mathbb{S}^{m-1})}+ \beta \| \nabla g \|_{L^2(\mathbb{S}^{m-1})} \quad \text{    s.t.  }\ \  \|v_{g,k} \|_{L^2(B)}=1,
\end{equation*}
where $\beta>0$ denotes the regularization parameter. 
In the numerical part, the regularization parameter is chosen as $\beta=10^{-2}$ and the radius of domain $B$ is given by $r=3$.
To test the stability, we add extra 1\% noise to the far field data associated with Neumann eigenvalues.
For comparison, we use the finite element method to compute the iterior Neumann eigenvalue problem \eqref{eq:inteNeu} and obtain the eigenfunctions $u_k$ with different eigenvalues, see the top row of Figure \ref{fig:eigenfunction}. The bottom row of Figure \ref{fig:eigenfunction} present the recovered Herglotz wave equations with recovered Neumann eigenvalues by using the gradient total least square method.
Here the dashed white lines denote the exact pear-shaped domain.
One can observe that the recovered Herglotz wave functions  are very close to the
exact Neumann eigenfunctions inside the obstacle.

\subsection{Reconstruct shapes}
In this section, we provide several two- and three-dimensional numerical examples to verify the efficient of Algorithm \ref{eq:alg1} and Algorithm \ref{eq:alg2}.

\subsubsection{2D reconstructions}
For the two-dimensional case, we choose an approximation of the boundary with the form
\begin{equation*}
	z(\phi)=\{r(\phi)\,(\cos \phi,\  \sin \phi): \phi\in[0, 2\pi]\},
\end{equation*}
where $r\in C^2([0,2\pi])$ is the radial function and it is given by trigonometric polynomials of order less than or equal $N_z\in \mathbb{N}$, i.e.,
\begin{equation*}
	r(\phi)=a_0+ \sum_{j=1}^{N_z} \left(a_j \cos j \phi+  b_j \sin j\phi\right), \quad \phi=[0,2\pi].
\end{equation*}
Here $a_0$, $a_1, a_2,..., a_{N_z}$ and $b_1, b_2,..., b_{N_z}$ are unknown Fourier coefficients.
In what follows, the order is chosen as $N_z=20$ and the stop criterion is set to be  $\|h\|_{L^2([0,2\pi])}<10^{-5}$. Moreover, the regularized parameter $\alpha$ is given by   $\alpha=10^{-5}$.
In what follows, the solid black lines denote the exact sound-hard obstacle, the dotted grey lines denote the initial guess and the red dashed lines denote the recovered shapes.

\begin{figure}
	\centering
	\subfigure[$R=1.1$]{\includegraphics[width=0.45\textwidth]
		{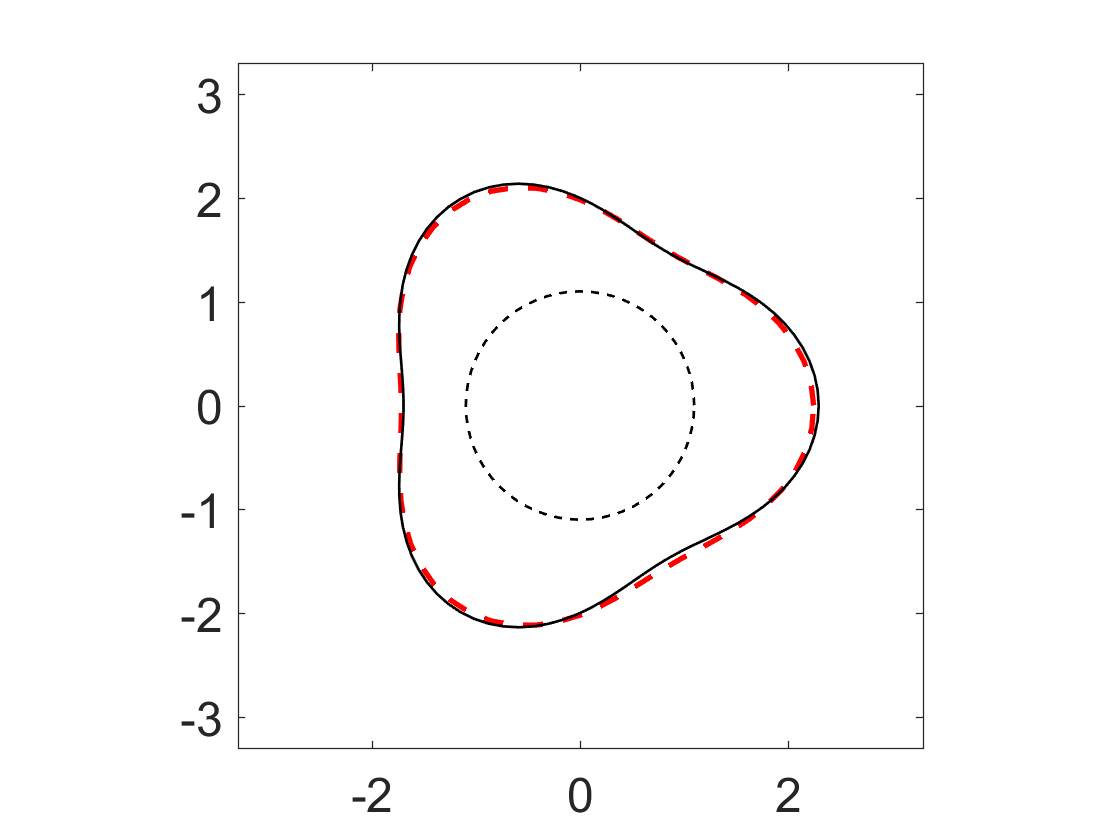}}
	\subfigure[$R=2.8$]{\includegraphics[width=0.45\textwidth]
		{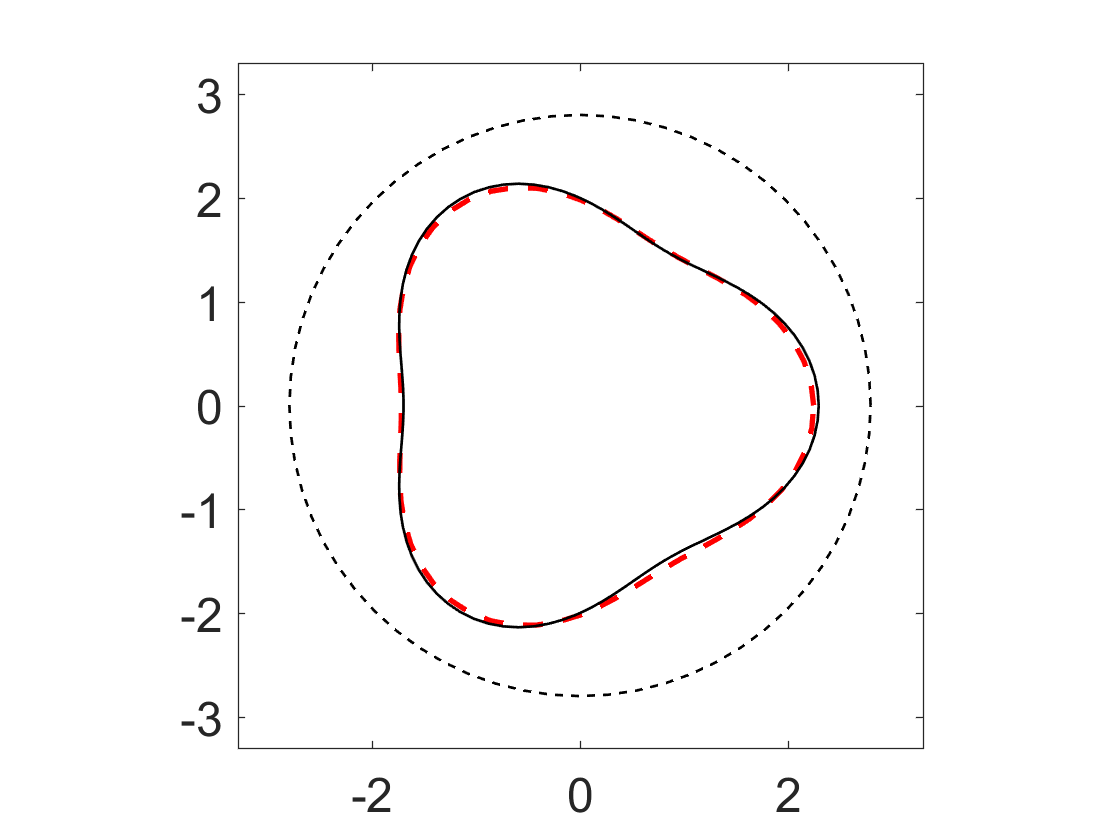}}\\
	\subfigure[$R=1.2$]{\includegraphics[width=0.45\textwidth]
		{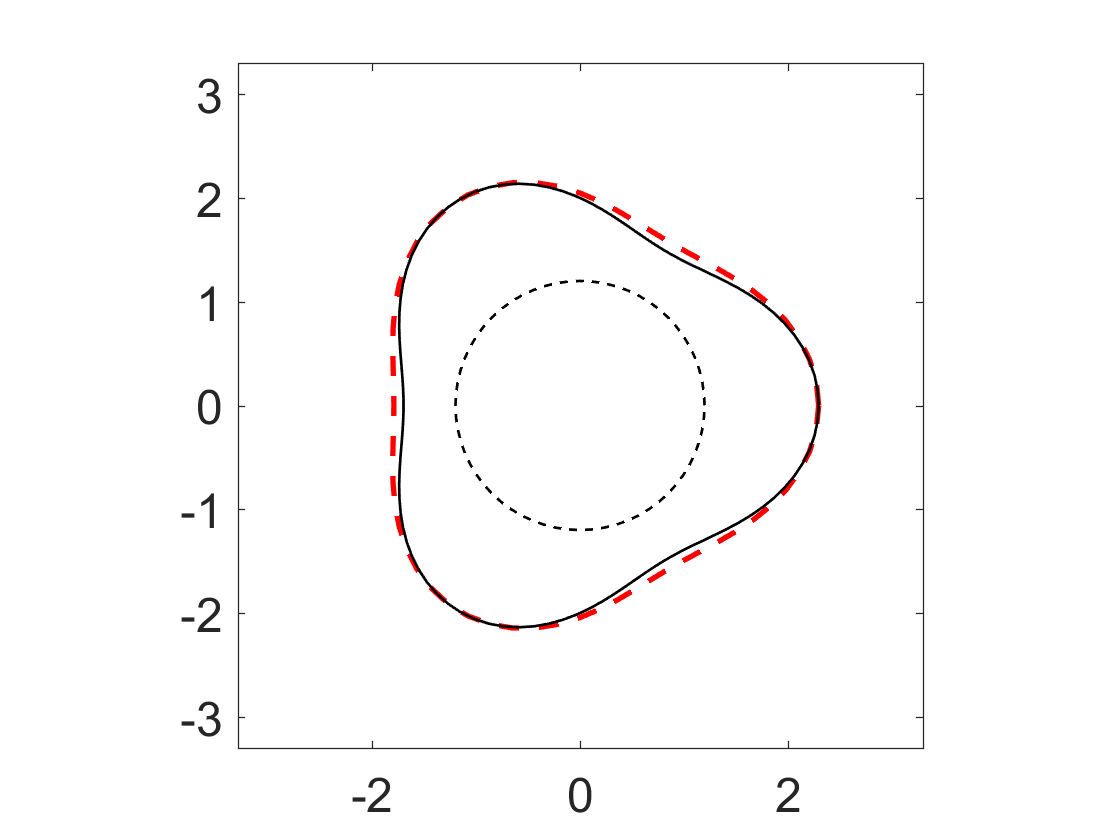}}
	\subfigure[$R=2.4$ ]{\includegraphics[width=0.45\textwidth]
		{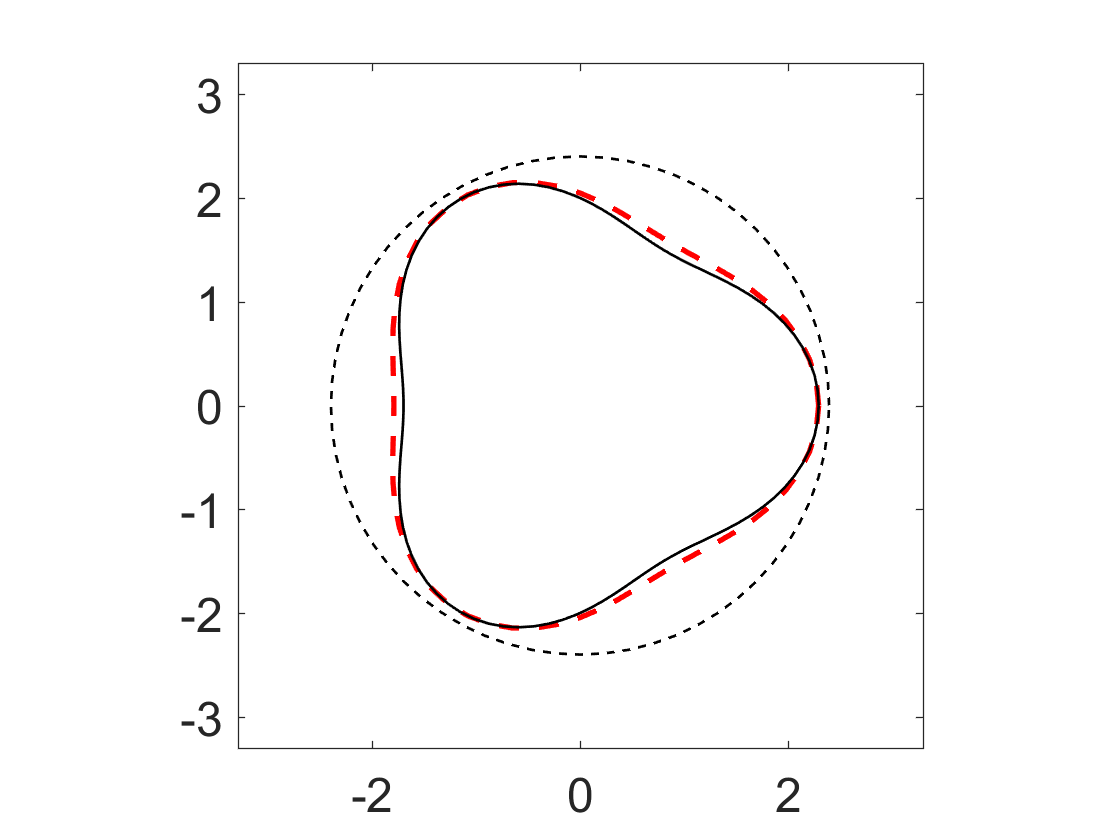}}\\
	\caption{\label{fig:pear-single} Reconstructions of the pear-shaped obstacle with different initial guesses $R$.
	}
\end{figure}

In the first example, we consider the pear-shaped domain as shown in \eqref{eq:pear}.
Since the initial value plays an important role for the Newton iterative method, we test the proposed Algorithm \ref {eq:alg1} with different initial guesses. Let the initial shape be a
circle centered at the origin  with a radius of $R$.
Figure \ref{fig:pear-single} presents  the reconstructions of the pear-shaped obstacle with different initial guesses.
It is noted that the stopping criteria is achieved between $15$ and $20$ iterations for this example.
Figure \ref{fig:pear-single}\,(a) and (b) respectively shows the recovery results for the minimum and maximum radius  with Neumann eigenvalue $k_1=1.562$. Correspondingly,  Figure 3(c)
and (c) respectively  shows the reconstructed results for the minimum and maximum radius with
Neumann eigenvalue $k_2= 1.708$. Through the numerical experiments, we find that for each Neumann eigenvalue there exists a interval of the radius such that the Newton iteration is convergent.
Moreover, we  test the multifrequency Newton-type method for recovering the shape. Figure \ref{fig:pear-multi} presents the reconstructed shapes via Algorithm \ref {eq:alg2}, where $R=1$ is the minimum radius and $R=3.2$ is the maximum radius. Here, we use four different Neumann eigenvalues. Comparing figure \ref{fig:pear-single} and \ref{fig:pear-multi}, one can observe that the multiple-frequency Newton-type approach has larger convergence range than single-frequency Newton-type approach.

\begin{figure}
	\centering
	\subfigure[$R=1$]{\includegraphics[width=0.45\textwidth]
		{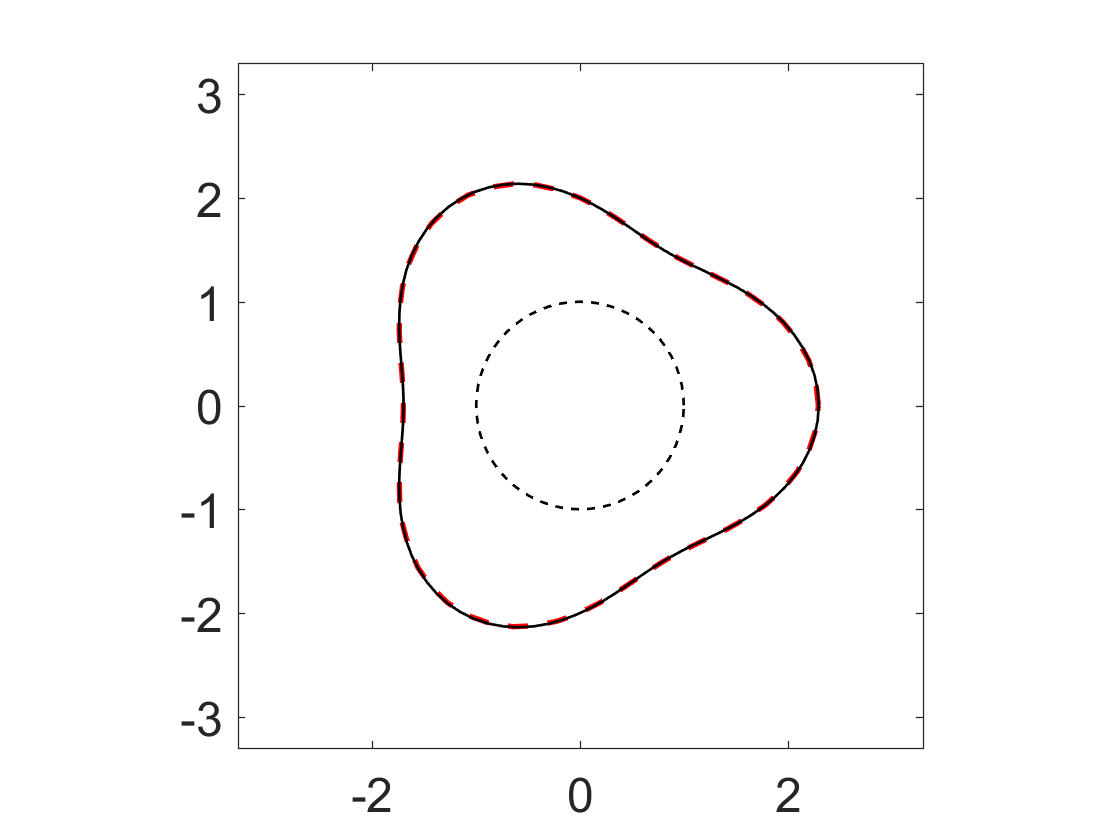}}
	\subfigure[$R=3.2$]{\includegraphics[width=0.45\textwidth]
		{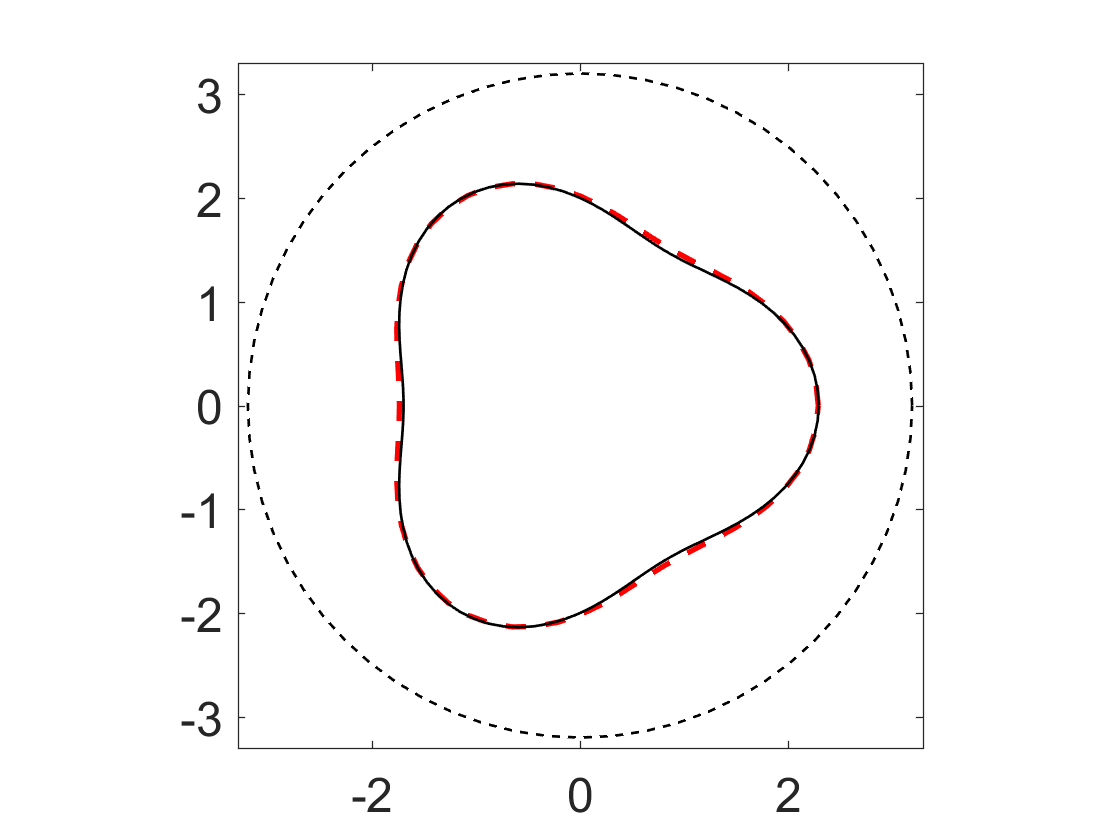}}
	\caption{\label{fig:pear-multi} Reconstructions of the pear-shaped obstacle by using Algorithm \ref {eq:alg2} with different initial guesses $R$.}
\end{figure}

\begin{figure}
	\centering
	\subfigure[$1\%$ noise]{\includegraphics[width=0.45\textwidth]
		{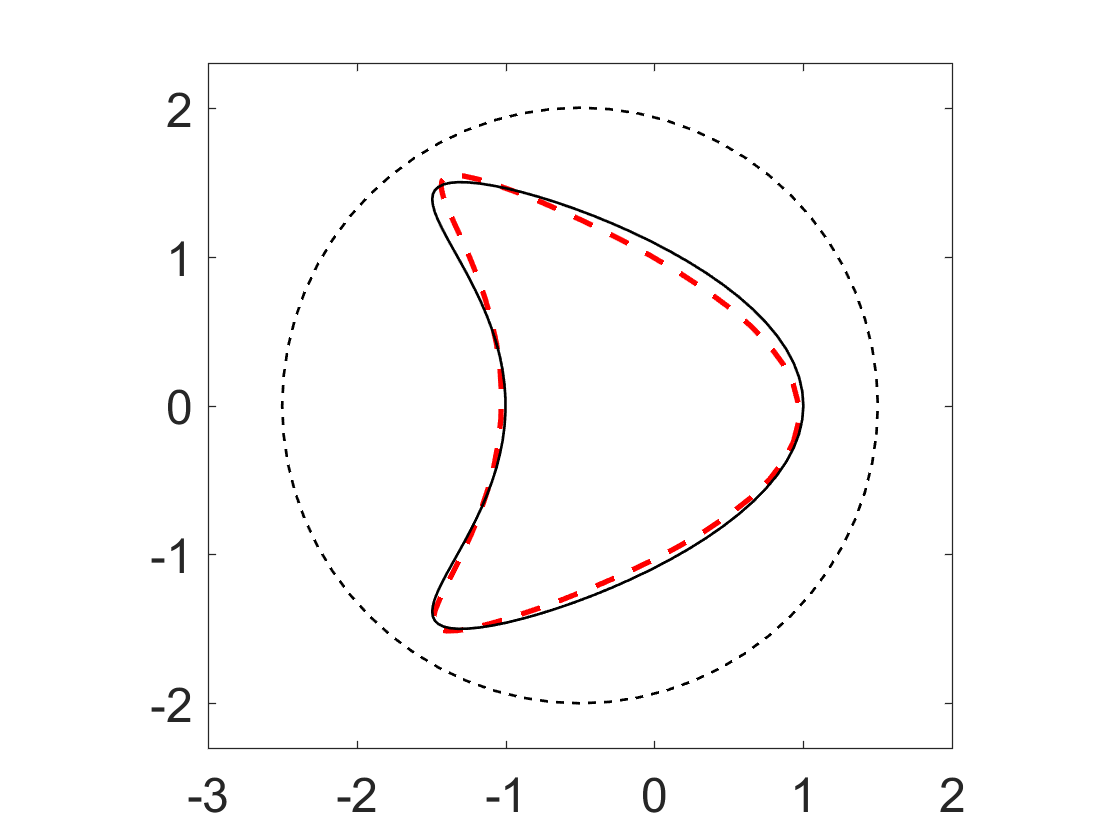}}
	\subfigure[$5\%$ noise]{\includegraphics[width=0.45\textwidth]
		{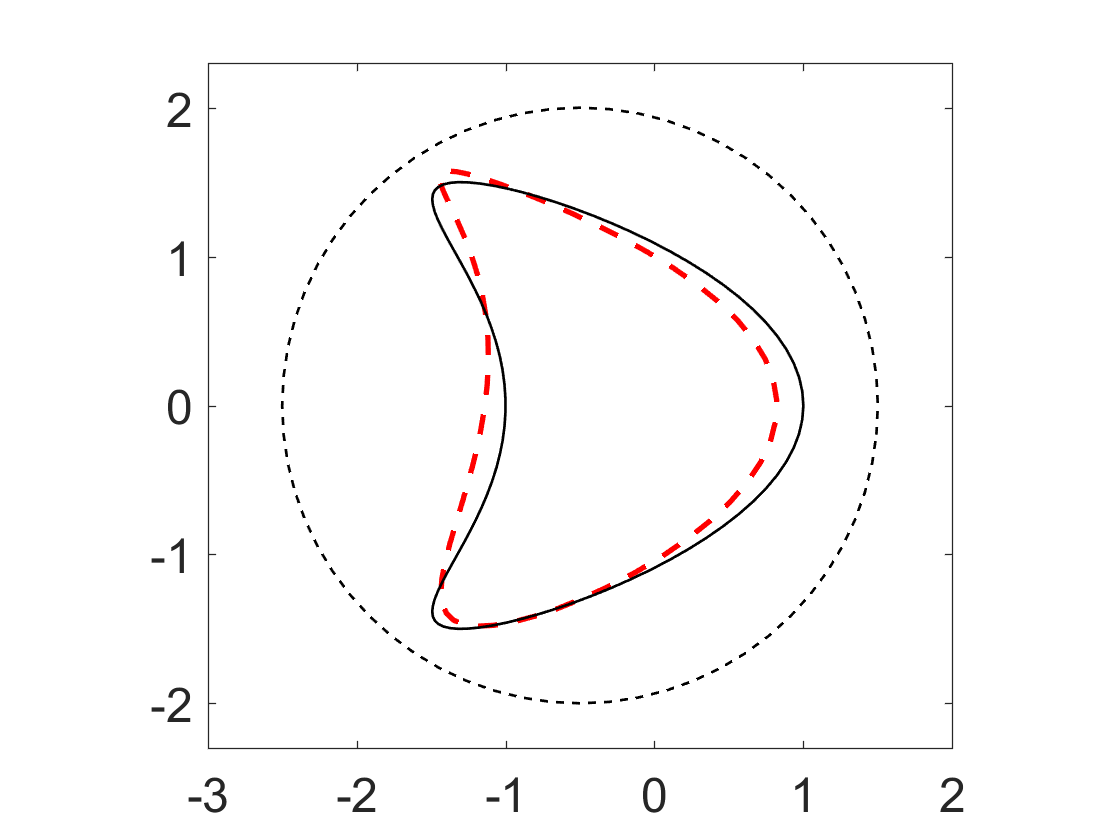}}\\
	\subfigure[$1\%$ noise]{\includegraphics[width=0.45\textwidth]
		{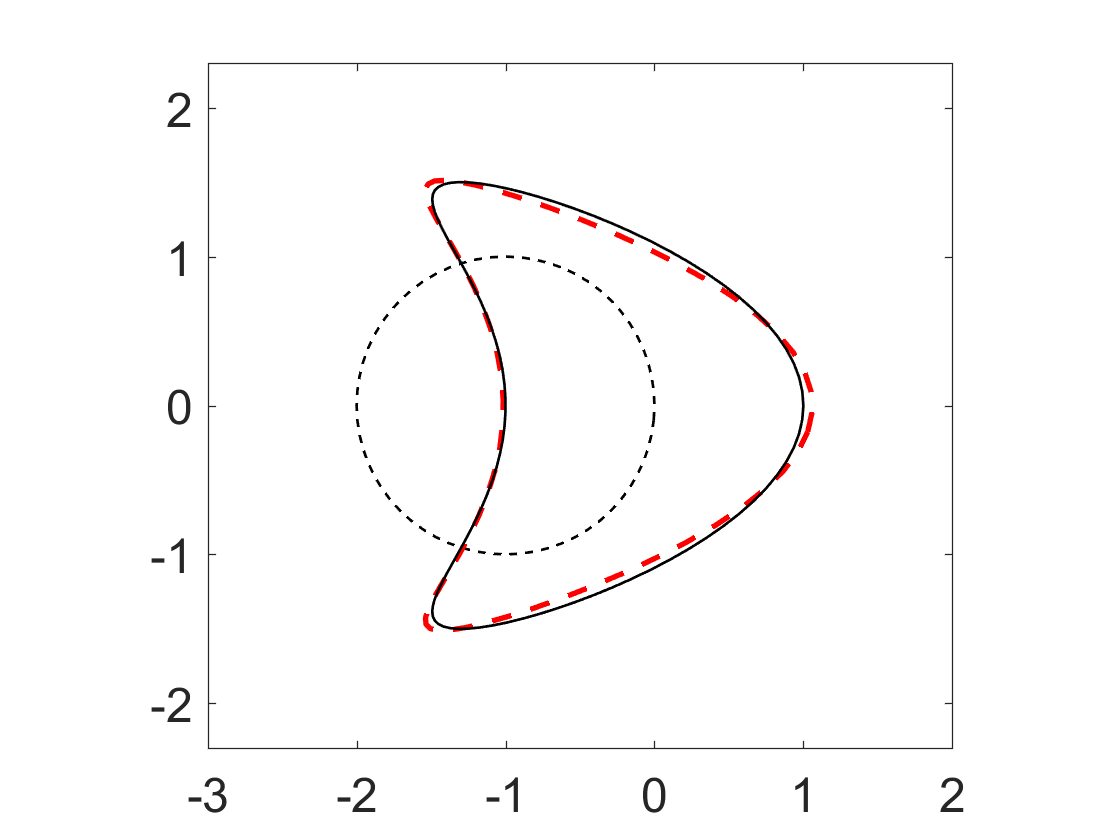}}
	\subfigure[$5\%$ noise]{\includegraphics[width=0.45\textwidth]
		{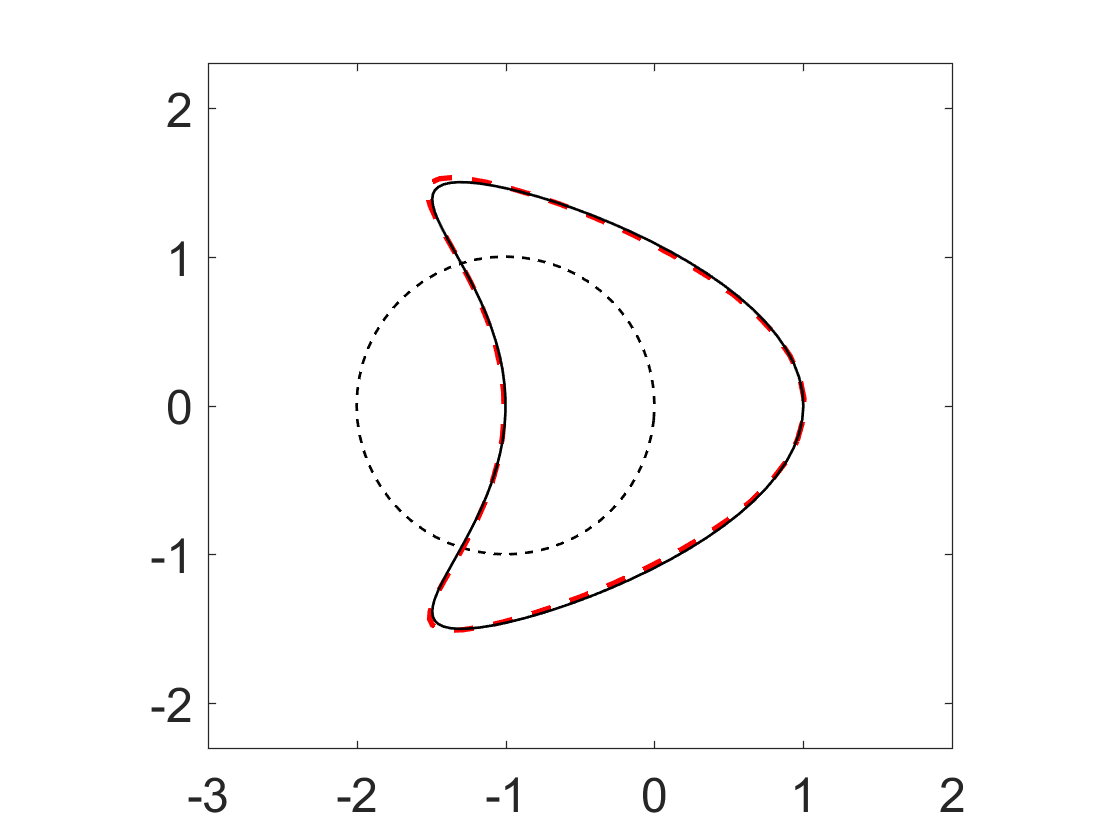}}
	\caption{\label{fig:kite} Reconstructions of the kite-shaped scatterer by using Algorithm \ref {eq:alg2}. Top row: reconstructions with three eigenvalues;  bottom row: reconstructions with six eigenvalues. }
\end{figure}

In the second example, we consider a concave case and the scatterer is given by a kite-shaped domain, which is parameterized by
\begin{equation*}
	x(\phi)=(\cos\phi +0.65 \cos 2\phi -0.65,\ 1.5 \sin \phi  ), \quad \phi\in[0,\, 2\pi].
\end{equation*}
Figure \ref{fig:kite} presents reconstructions of the kite-shaped scatterer by using the multifrequency Newton-type method, i.e., Algorithm \ref{eq:alg2}. The top row of figure \ref{fig:kite} shows the recovery results with the first three eigenvalues, where the initial shape is a circle centered at $(-0.5,0)$ with radius $R=2$. The bottom row of figure \ref{fig:kite} shows the recovery results with the first six eigenvalues, where the initial guess is a circle centered at $(-1,0)$ with radius $R=1$. To test the stability of the proposed approach, $1\%$ and $5\%$ noise are respectively added to the measured far-field data. From figure  \ref{fig:kite} (a) and (b), it is clear to see that the accuracy is reduced as the noise increases. Furthermore, according to figure \ref{fig:kite} (b) and (d), one can find that the accuracy is improved as the number of frequencies increases.

\subsubsection{3D reconstructions}
In this part, we consider a more challenging case for reconstructing the three-dimensional scatterer.  For the three-dimensional case, we choose an approximation of the surface boundary with the form
\begin{equation*}
	z(\theta, \phi)=\{r(\theta, \phi)\,(\sin\theta \cos \phi,\,  \sin\theta \sin \phi,\,  \cos \theta): \theta\in[0,\pi],\, \phi\in[0, 2\pi]\},
\end{equation*}
where $r\in C^2([0,\pi]\times[0,2\pi])$ is the radial function and it is given by spherical harmonics with order less than or equal $N_z\in \mathbb{N}$, i.e.,
\begin{equation*}
	r(\theta, \phi)= \sum_{\ell=0}^{N_z} \sum_{s=-\ell}^{\ell}  \left( a_{\ell}^s \cos s \phi+ b_{\ell}^s \sin s \phi\right)P_{\ell}^s(\cos \theta).
\end{equation*}
In what follows, the regularized parameter $\alpha$ is given by $\alpha=10^{-4}$, the stop criterion is set to be  $\|h\|_{L^2([0,\pi]\times[0,2\pi])}<10^{-4}$ and the order is chosen as $N_z=8$. In addition,  the measurement directions are chosen to be $500$ pseudo-uniformly distributed measured points on the unit sphere. The incident directions are chosen to be $15\times 30$ uniform rectangular mesh of $[0,\pi]\times[0,2\pi]$.

We first consider an ellipsoid domain in three dimensions (see Figure \ref{fig:elipsolid}(a) ), which is parameterized by
\begin{equation*}
	x(t,\tau)=(3\sin t \cos \tau,\, 3\sin t \sin \tau,\, 6\cos t ), \quad t\in[0,\,\pi],\ \tau\in [0,\, 2\pi].
\end{equation*}
Here $1\%$ perturbed noise is added to the far field data  and we use Algorithm \ref{eq:alg1} to recover the shape of the obstacle.  Figure \ref{fig:elipsolid} shows the reconstructed shapes with different initial guesses, where the eigenvalue is chosen as $k=1.4513$. From the surface plots, it can be seen that the reconstructions are very close to the exact obstacle.

Next, we are devoted to the identification of a concave scatter, where the domain is parameterized
by
\begin{equation*}
	x(t,\tau)=(1.5 \sin t \cos \tau,\, 1.5\sin t \sin \tau,\, 0.2-\cos t-0.65 \cos 2t ),
\end{equation*}
for $t\in[0,\,\pi]$ and $\tau\in [0,\, 2\pi]$.
Actually, it is difficult to determine the boundary of the obstacle by using single frequency data. Therefore, the multifrequency Newton-type method is used to recover the shape with two Neumann eigenvalues ($k_1=2.4514$ and $k_2=3.2630$).
To exhibit the iterative process, we plot recovery results with different iteration numbers $N_t$ in figure \ref{fig:kite3D}. One can observe that the proposed approach demonstrates better imaging performance as the number of iterations increases.

\begin{figure}
	\centering
	\subfigure[exact]{\includegraphics[width=0.45\textwidth]
		{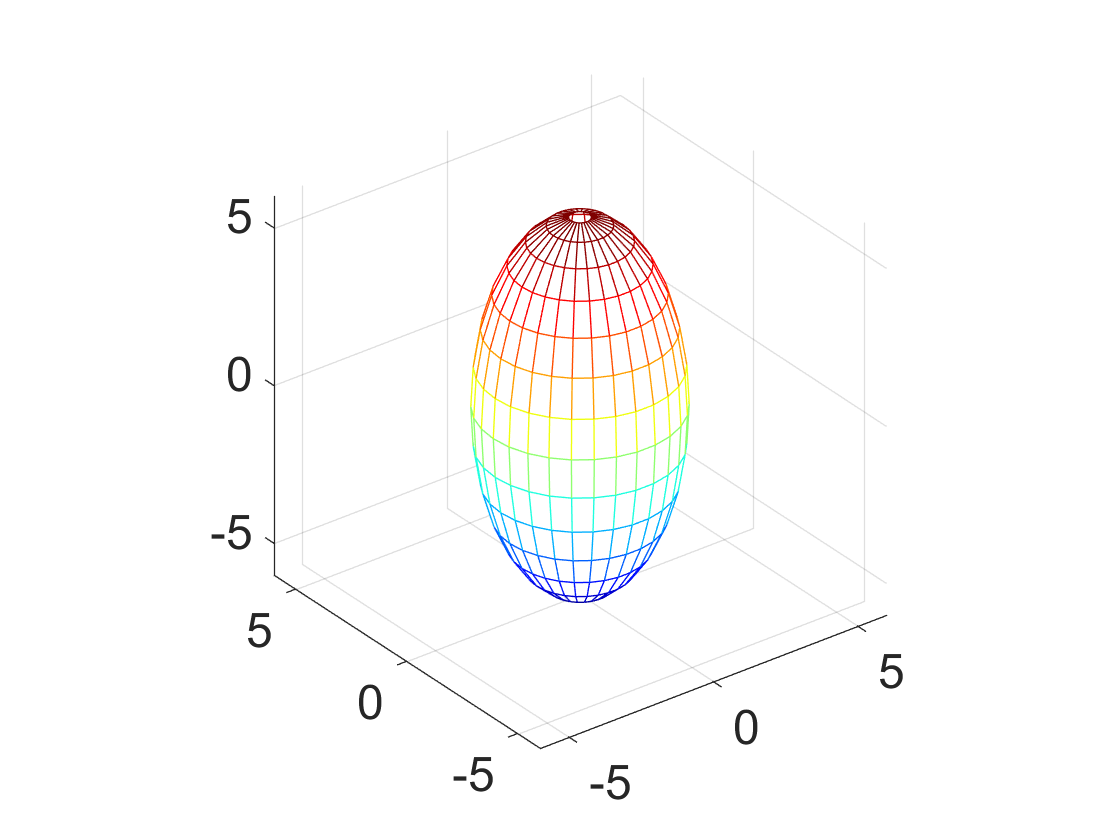}}
	\subfigure[recovery with $R=3$]{\includegraphics[width=0.45\textwidth]
		{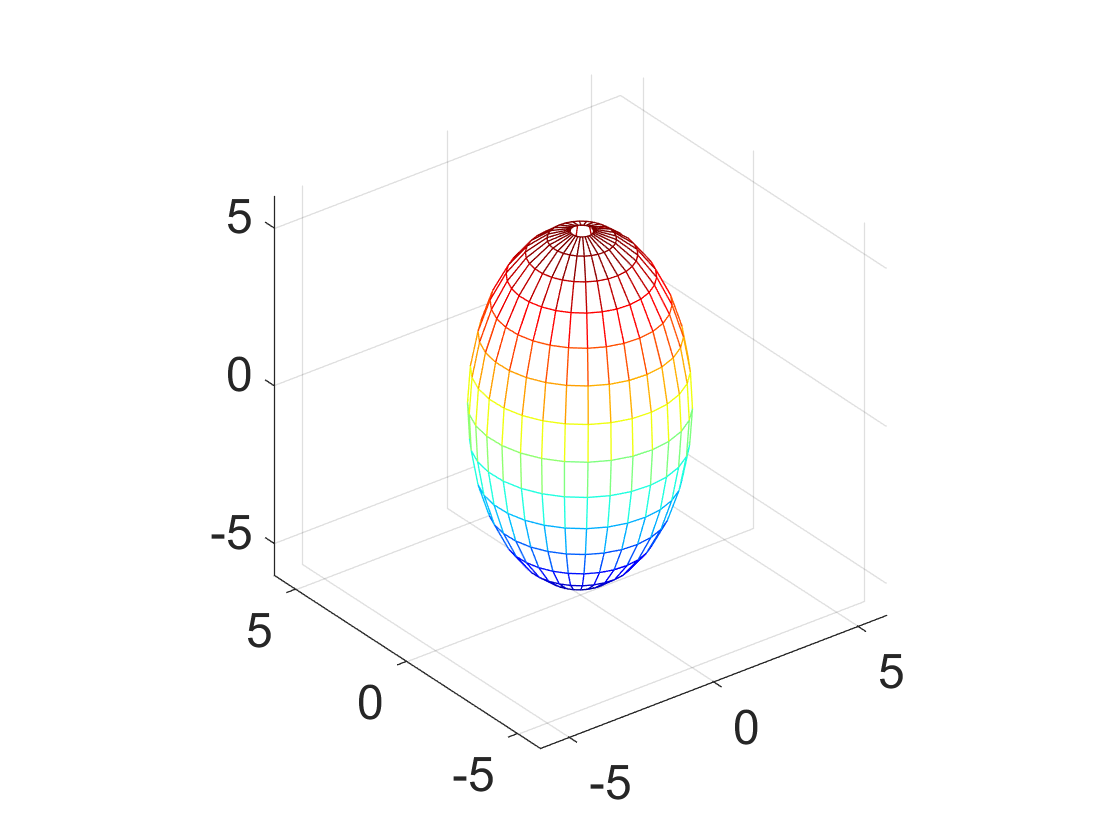}}\\
	\subfigure[recovery with $R=4$]{\includegraphics[width=0.45\textwidth]
		{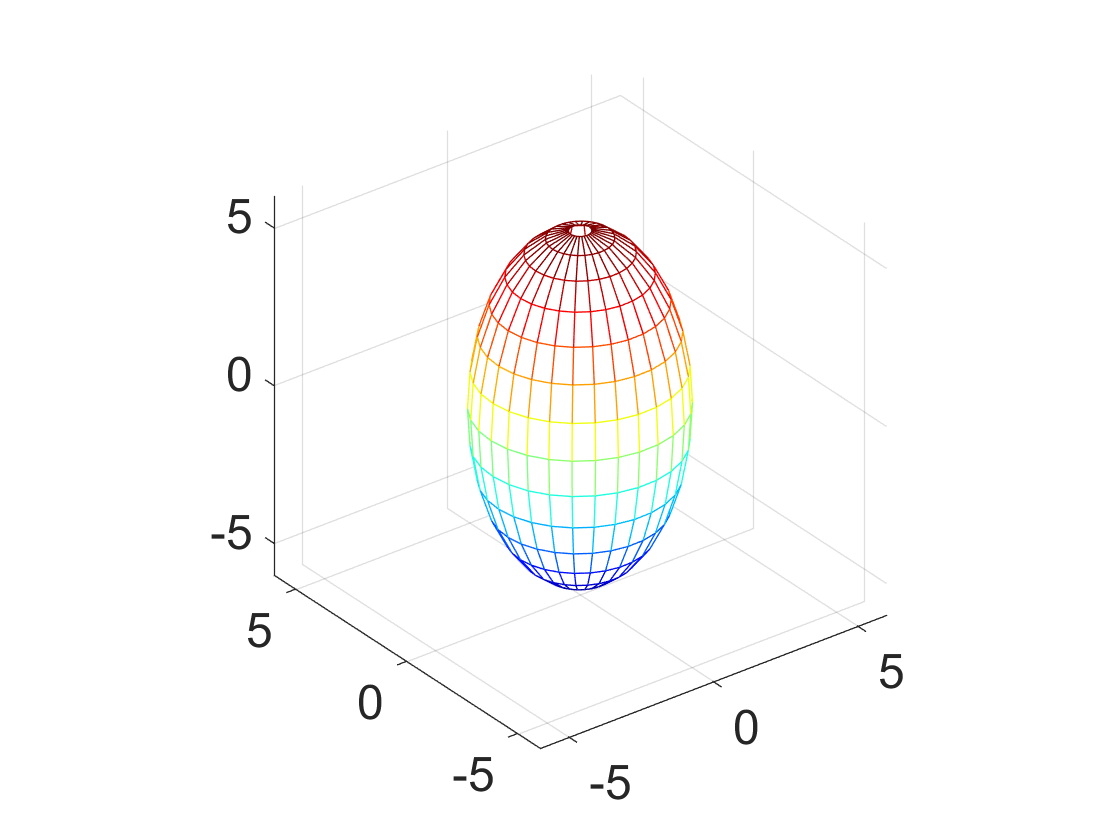}}
	\subfigure[recovery with $R=5$]{\includegraphics[width=0.45\textwidth]
		{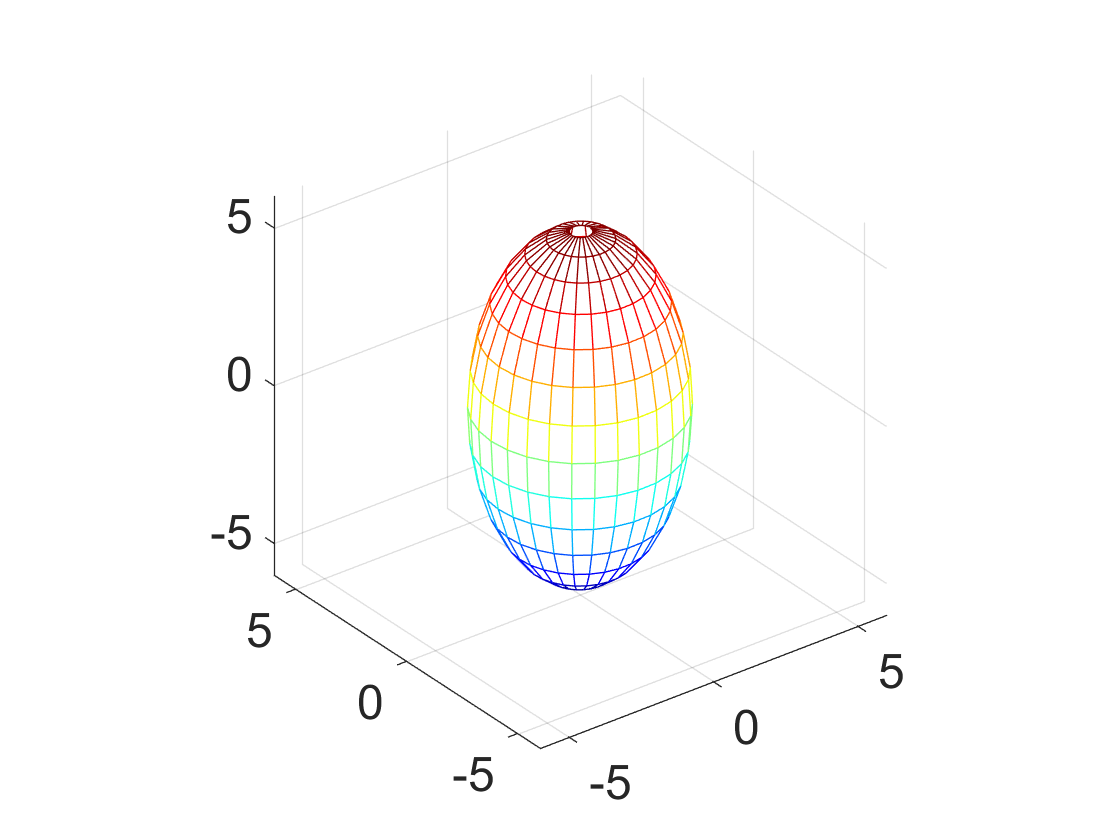}}
	\caption{\label{fig:elipsolid} Exact and reconstructed ellipsoid obstacles with different initial guesses $R$.}
\end{figure}

\begin{figure}
	\centering
	\subfigure[exact]{\includegraphics[width=0.45\textwidth]
		{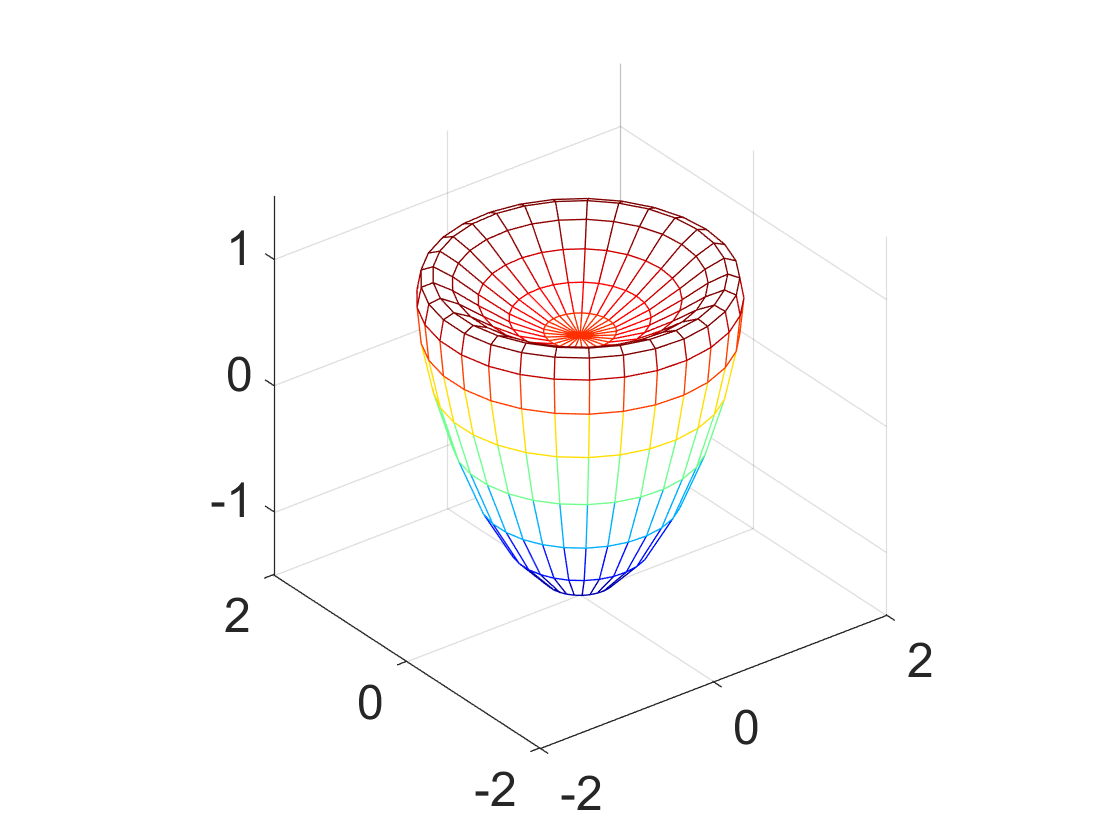}}
	\subfigure[$N_t=1$]{\includegraphics[width=0.45\textwidth]
		{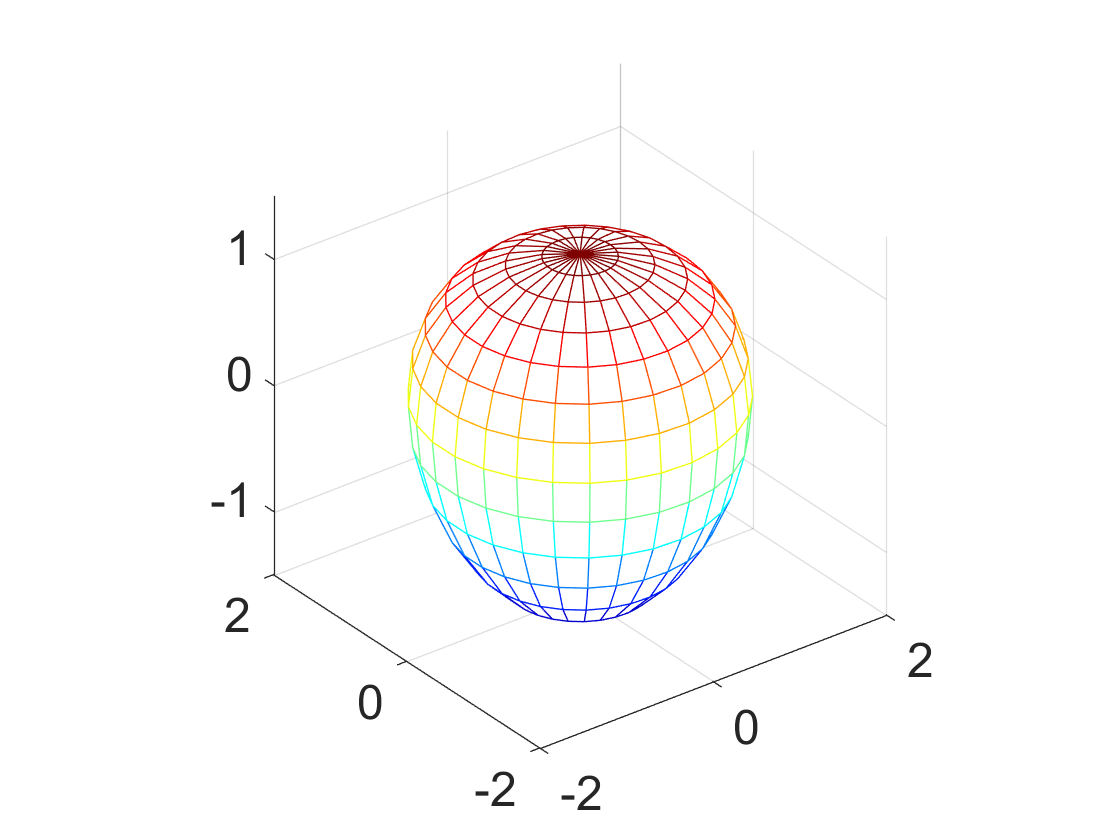}}\\
	\subfigure[$N_t=3$]{\includegraphics[width=0.45\textwidth]
		{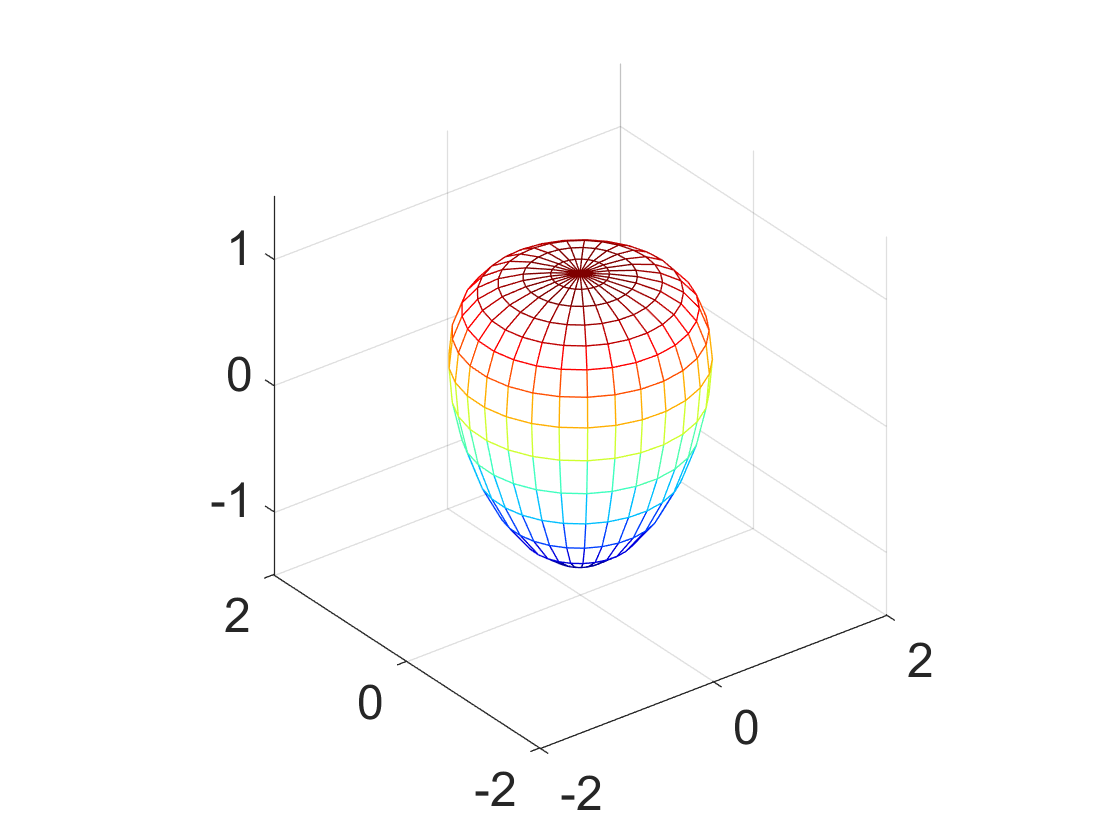}}
	\subfigure[$N_t=5$]{\includegraphics[width=0.45\textwidth]
		{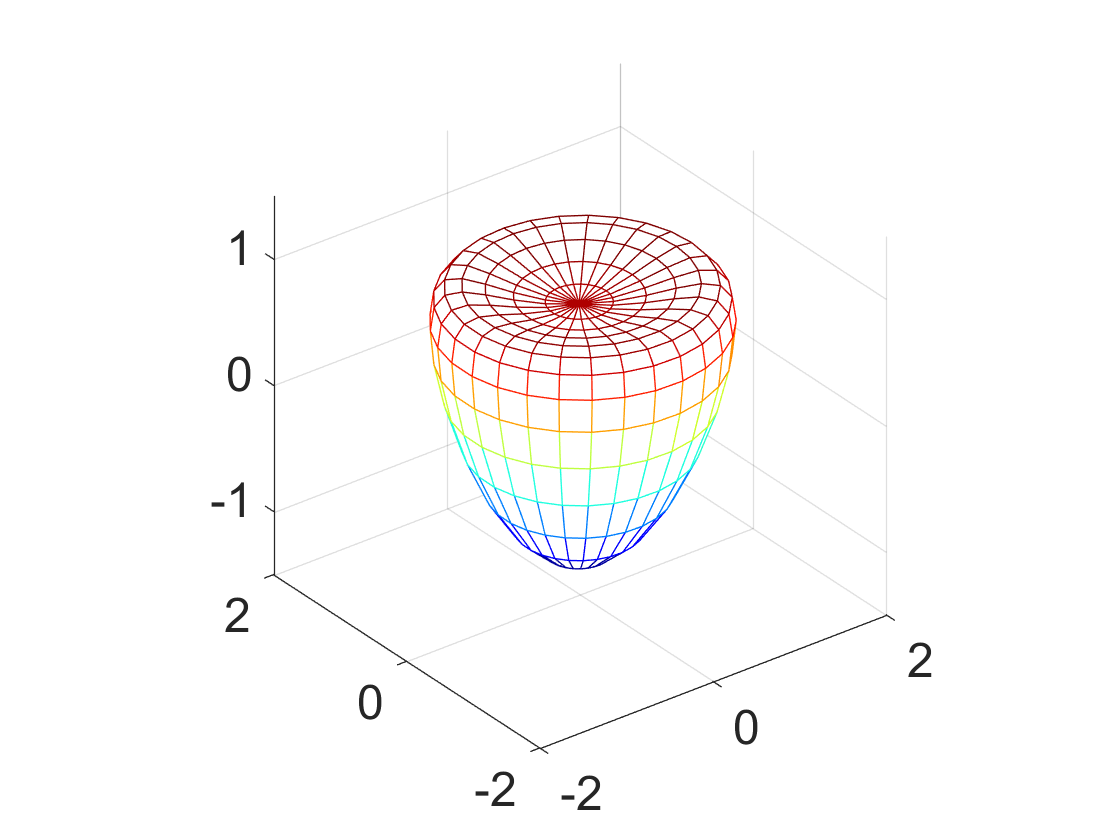}}\\
	\subfigure[$N_t=7$]{\includegraphics[width=0.45\textwidth]
		{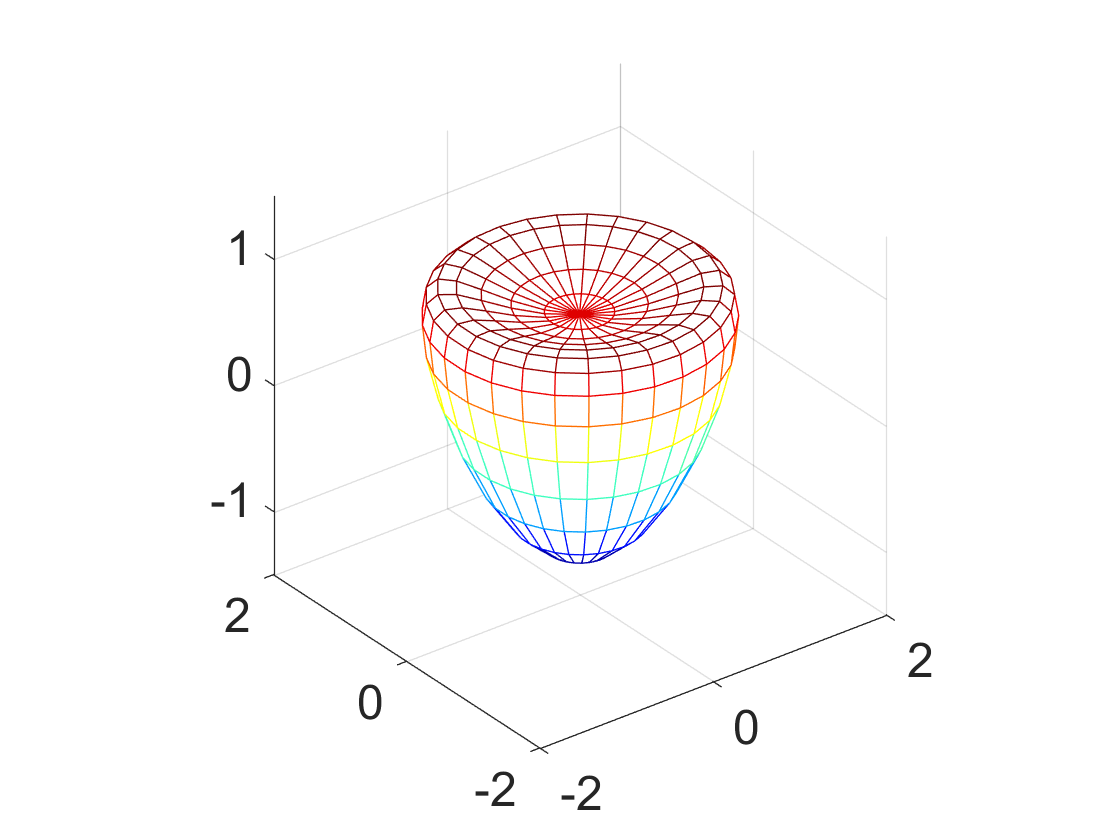}}
	\subfigure[$N_t=10$]{\includegraphics[width=0.45\textwidth]
		{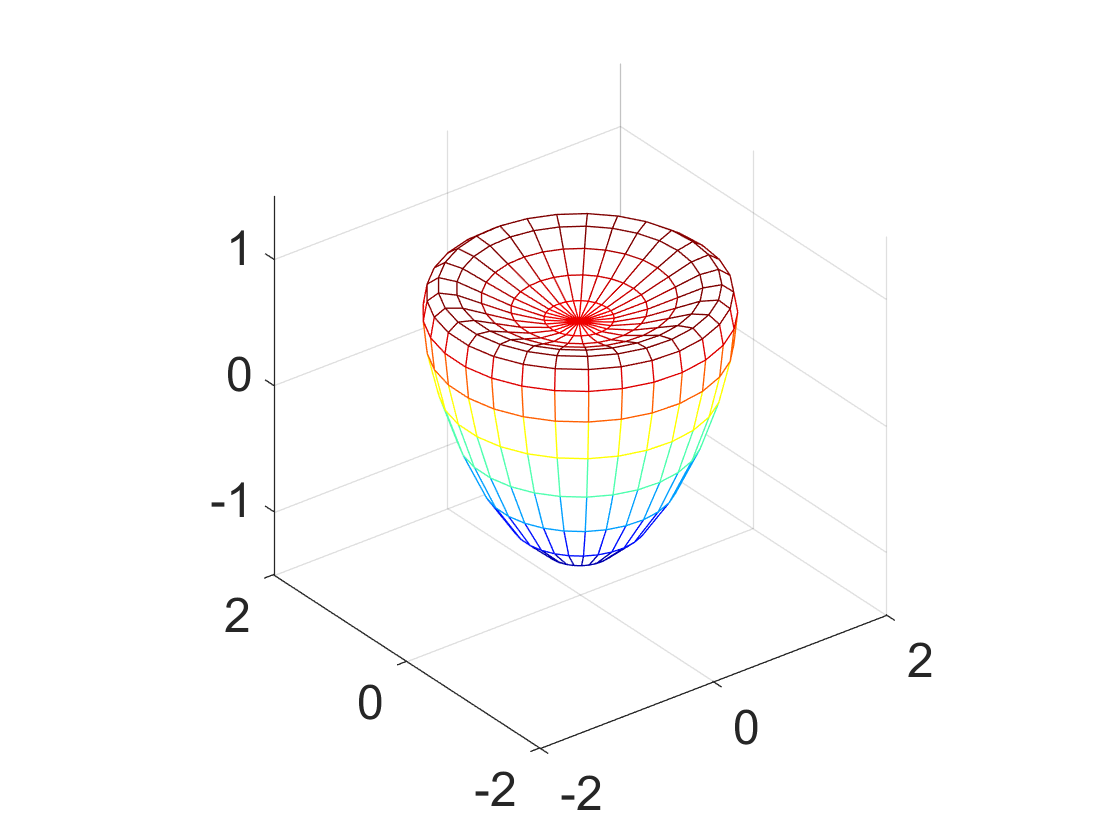}}
	\caption{\label{fig:kite3D} Surface plots of the exact and reconstructed results with different iteration numbers $N_t$.}
\end{figure}

\begin{rem}
	In order to ensure the consistency of theory and numeric, we only present numerical examples for reconstructing the sound-hard obstacle. In fact, the proposed imaging approach based on the interior resonant modes is easily extended to recover the sound-soft obstacle.
\end{rem}

\section*{Acknowledgment}
The work of H. Liu was supported by Hong Kong RGC General Research Funds (project numbers, 11300821, 12301420 and 12302919) and the NSFC-RGC Joint Research Grant (project number, N\_CityU101/21).
The work of X. Wang was supported by the NSFC grants (12001140 and 11971133).

\end{document}